\documentclass [11pt]{amsart}
\usepackage{geometry,braket}
\newgeometry{margin=1.35in}
\usepackage[utf8]{inputenc}
\usepackage{amsfonts, amsmath, amssymb, amsthm,mathtools, stmaryrd, blkarray,bm, bbm, ytableau}
\usepackage{verbatim}
\usepackage[normalem]{ulem}
\usepackage{tikz-cd}
\usepackage{enumitem}
\usepackage{esvect}
\usepackage{float}

\usepackage{array}

\usepackage{hyperref}

\theoremstyle{theorem}
\newtheorem{theorem}{Theorem}[section]
\newtheorem{corollary}[theorem]{Corollary}
\newtheorem{lemma}[theorem]{Lemma}
\newtheorem{proposition}[theorem]{Proposition}

\numberwithin{equation}{section}

\theoremstyle{definition}
\newtheorem{definition}[theorem]{Definition}

\newtheorem{remark}[theorem]{Remark}

\theoremstyle{remark}
\newcommand{\BC}{{\mathbb{C}}}

\newcommand{\BN}{{\mathbb{N}}}

\newcommand{\BP}{{\mathbb{P}}}
\newcommand{\BQ}{{\mathbb{Q}}}
\newcommand{\BR}{{\mathbb{R}}}

\newcommand{\BZ}{{\mathbb{Z}}}

\newcommand{\CA}{{\mathcal A}}

\newcommand{\CE}{{\mathcal E}}
\newcommand{\CF}{{\mathcal F}}

\newcommand{\CM}{{\mathcal M}}

\newcommand{\CO}{{\mathcal O}}

\newcommand{\CX}{{\mathcal X}}

\newcommand{\CZ}{{\mathcal Z}}

\newcommand{\Fz}{{\mathfrak{z}}}

\DeclareFontFamily{OT1}{rsfs}{}
\DeclareFontShape{OT1}{rsfs}{n}{it}{<-> rsfs10}{}
\DeclareMathAlphabet{\curly}{OT1}{rsfs}{n}{it}

\newcommand{\Aut}{\operatorname{Aut}}
\newcommand{\p}{\mathbb{P}}

\newcommand{\Mbar}{{\overline M}}

\newcommand{\GW}{\mathsf{GW}}

\newcommand\ev{\operatorname{ev}}

\allowdisplaybreaks
\begin{document}
	\baselineskip=14.5pt
	
	\title{Admissible covers and stable maps}

	\author{Denis Nesterov, Maximilian Schimpf,\\
		An appendix by Johannes Schmitt}

	\begin{abstract}
		Moduli spaces of admissible covers  and stable maps of target curves give rise to cycles on $\Mbar_{g,n}$. We prove a formula relating these cycles. It recovers both the Ekedahl--Lando--Shapiro--Vainshtein formula and the Gromov--Witten/Hurwitz correspondence, providing a cycle-theoretic refinement thereof.  The formula is verified computationally in low-genus cases with the help of Johannes Schmitt. 
	\end{abstract}

	\maketitle
	
	\setcounter{tocdepth}{1} 
	\tableofcontents
	
	\section{Introduction}
	
	\subsection{Results} Fix an integer $d \in \BN$ and a vector of partitions $\underline{\eta}=(\eta^1,\hdots, \eta^k)$ of $d$.  Hurwitz cycles are associated to the loci of  stable curves admitting a ramified  cover  to a genus $h$  curve with the ramification profile prescribed by $\underline{\eta}$, 
	\[ 
	\mathsf{H}_{ \underline{\eta}} \in H^*(\Mbar_{g_0, \sum \ell(\eta^i)}),
	\] 
	such that $g_0$ is determined by the Riemann--Hurwitz formula applied to $h$ and $\underline{\eta}$, while ramification points are marked. For singular curves, we consider admissible covers of Harris and Mumford \cite{HM82}, which provide irreducible compactifications of moduli spaces of ramified covers to smooth targets. 
	
	Gromov--Witten cycles can be viewed as associated to the loci of genus $g$ curves admitting a degree $d$ stable map\footnote{A map is stable, if it is fixed by at most finitely many automorphisms of target and source curves.} to a genus $h$ curve, 
	\[ 
	\mathsf{GW}_{g} \in H^*(\Mbar_{g}),
	\vspace{-0.05cm}\] 
	defined more precisely as pushforwards of virtual fundamental classes of moduli spaces of stable maps introduced by Kontsevich \cite{Kon}. Both Hurwitz and Gromov--Witten cycles admit natural generalisations  by introducing marked points, relative points,  capping with various classes, etc; this is discussed in detail in Section \ref{sectionCycle}.

	We establish  a formula,  Theorem \ref{maintheorem},  that expresses $\mathsf{GW}_{g}$ in terms of 	$\mathsf{H}_{ \underline{\eta}}$ for different $\underline{\eta}$ and certain correction classes that account for degeneracies of stable maps, i.e., contracted components and ramifications. 
	These correction classes are given by expressions involving $\lambda$- and $\psi$-classes  that appear in the localisation formula on $\p^1$, and referred to as $I$-functions, Definition \ref{defnI}. Our formula holds in the most general setting. 
	Applications of the formula considered in this work include: 
	
	\begin{itemize}
		\item[(1)] A more effective computation of  Torelli pullbacks of $[\CA_1\times \CA_{g-1}]$ from moduli spaces of abelian varieties $\CA_g$,  considered in \cite{COP}. 
		\item[(2)] A recursive formula for cycles of the Hyperelliptic loci, reproving a result of \cite{FabP} that these cycles are tautological. 
		\item[(3)]A different proof of the Ekedahl--Lando--Shapiro--Vainshtein (ELSV) formula of \cite{ELSV}. 
	\end{itemize}
	The first two results correspond to degree $d=1$ maps to the moving elliptic curve and degree $d=2$ maps to  $\p^1$ with two relative points, respectively. In Appendix \ref{AppendixJohannes} written by Johannes Schmitt, we give a low-genus computational verification of the formula in the case of degree $d=1$ and $d=2$ maps to $\p^1$ with two relative points, using the  \texttt{admcycles} package \cite{admcycles}. The third result follows from taking the minimal power of $z$ in the formula associated to maps to $\p^1$ with one relative point, Theorem \ref{ELSV}. 
	
	 By intersecting cycles  with $\psi$-classes, we obtain a numerical version of the formula, Corollary \ref{corollary1},  which leads to a further application: 
	
	\begin{itemize}
		
		\item[(4)] A different proof of the Gromov--Witten/Hurwitz correspondence of \cite{OPcom}. 
	\end{itemize}

	This work grew out of an attempt to give a more geometric perspective on the Gromov–Witten/Hurwitz correspondence. The correspondence concerns only descendent point insertions, however, our formula expresses arbitrary descendent insertions of a curve in terms of Hurwitz numbers, Fulton--MacPherson integrals and Hodge integrals. In particular, all of the aforementioned invariants are explicitly computable. Moreover, this also holds for a moving curve. 
	
	Our formula can therefore be viewed as a cycle-valued refinement of both the ELSV formula and the Gromov--Witten/Hurwitz correspondence for arbitrary insertions. More importantly, it provides a geometric reason for any relation between  admissible covers and stable maps to hold. 
	
	%, which also makes the ELSV formula and the Gromov–Witten/Hurwitz correspondence consequences of the same result.
	
	\subsection{Methods} 
	Theorem \ref{maintheorem} is a cycle-valued wall-crossing formula associated to a change of stability provided by 
	$\epsilon$-unramification, defined in full generality in \cite{NuG} and first considered in \cite{N22} under the different name of 
	$\epsilon$-admissibility. We use the same master space techniques, based on Zhou's theory of entangled tails \cite{YZ}, to establish it. Instead of integrating classes, we push them forward to $\Mbar_{g,n}$. This leads to extra layers of complexity in the wall-crossing formula. In particular, already in the degree one, the formula provides a highly non-trivial relation in $H^*(\Mbar_{g})$, Corollary \ref{cor_one}.

	The proof of the Gromov–Witten/Hurwitz correspondence in Section \ref{sectionGWH} still utilizes the operator formalism from \cite{OPcom}. In particular, we need the results on analytic continuations of certain operators established in \cite{OPeq}. However, the 
	$I$-functions responsible for the wall-crossing formula admit more explicit expressions than the relative Gromov–Witten invariants considered in \cite{OPcom}, which allows us to uncover the completed cycles more effectively.
	
	\subsection{Acknowledgements} We are grateful to Samir Canning, Alessandro Giacchetto, Aitor Iribar Lopez, Aaron Pixton and Rahul Pandharipande for useful discussions on related topics. Special thanks to Alessandro for reading a preliminary version of the paper. 
	
	D.N.\ was supported by a Hermann-Weyl-instructorship from the
	Forschungsinstitut f\"ur Mathematik at ETH Z\"urich. M.S.\ was supported by the starting grant ``Correspondences in enumerative geometry: Hilbert schemes, K3 surfaces and modular forms", No 101041491 of the European Research Council. J.S.\ was supported by SNF-200020-219369 and SwissMAP.

	\section{Cycle-valued wall-crossing formula} \label{sectionCycle}
	\subsection{Admissible covers and stable maps} Let
	\[ \CX \rightarrow \Mbar^\circ_{h,m }\]
	be the universal curve over the moduli space of stable connected curves of genus $h$ with $m$ markings. Fix a degree $d \in \BN$ and consider  $\Mbar^\circ_{h,m }$-relative moduli spaces of stable maps and admissible covers of $\CX$ with possibly  disconnected sources, 
	\[ \Mbar_{g,n}(\CX,\underline{\mu}), \quad \overline{H}(\CX,\underline{\mu},\underline{\eta}), \]
	where $\underline{\mu}=(\eta^1, \dots, \eta^m)$ is a vector of partitions of $d$ which specifies ramification profiles of maps at universal markings of $\CX$, and  $\underline{\eta}=(\eta^1, \dots, \eta^k)$ is a vector of partitions of $d$, which specifies ramification profiles of admissible covers at additional  moving $k$ markings, i.e., over $\Mbar^\circ_{h,m+k}$. In both cases, branching points are ordered by the order of partitions in a vector, while ramification points have a standard order. In less technical terms, these spaces parametrize stable maps and admissible covers of a varying nodal curve of genus $h$ with specified ramification profiles at varying marked points. 
	
	Over nodes, we require maps to be admissible in the sense of \cite{HM82}. In the case of stable maps, we allow target curves to acquire rational bridges to ensure the properness of the moduli spaces. For admissible covers, both rational bridges and rational tails are allowed for the same reason. For brevity, we do not indicate these features explicitly in the notation of the corresponding spaces.
	
		\begin{remark}
		There are several ways to construct spaces $\Mbar_{g,n}(\CX,\underline{\mu})$ an $\overline{H}(\CX,\underline{\mu},\underline{\eta})$. The original constructions were proposed by Li \cite{Li} and Harris--Mumford \cite{HM82}, respectively. Certain (virtual) normalisations of those spaces are provided by considering orbifold \cite{ACV} or logarithmic curves \cite{Kim}. The pushforwards of the resulting (virtual) fundamental classes to moduli spaces of source curves are independent of the chosen construction. For the localization on the master space in \cite{NuG}, the logarithmic perspective was used. The reader is free to choose any preferred approach.
		\end{remark}	
	
		\subsection{Disconnected and connected curves}
	
	Throughout the article, $\Mbar_{g,n}$  will denote the space of possibly disconnected curves of arithmetic genus $g$ with $n$ marked points, while $\Mbar_{g,n}^\circ$ will denote its connected version. More explicitly, 
	\begin{equation}\label{disconnected}
		\Mbar_{g,n}= \coprod_{\{(g_i,N_i)\}_{i=1}^l} \prod^{l}_{i=1} \Mbar^\circ_{g_i,N_i},
	\end{equation}
	where the disjoint union runs over partitions $\{(g_i,N_i)\}_{i=1}^l$ satisfying
	\[
	\sum_{i=1}^l g_i-l+1 = g \text{ and } \coprod_{i=1}^l N_i = \{1,\ldots,n\}, \ N_i\neq \emptyset, 
	\] 
	and $\Mbar^\circ_{g_i,N_i}$  are moduli spaces of connected curves, whose marked points are labelled by elements of the set $N_i$. We permit $(g_i,|N_i|)$ to take values $(0,1)$ and $(0,2)$, defining $\Mbar^\circ_{g_i,|N_i|}$ to be a point in these cases, 
	\[ \Mbar^\circ_{g_i,N_i}= \mathrm{pt}, \quad  \text{ if } (g_i,|N_i|)=(0,1) \text{ or } (0,2), \]
	while factors of the form $\Mbar^\circ_{0,0}$ are excluded.

	\subsection{Gromov--Witten  cycle} For moduli spaces of stable maps, there exist the following natural morphisms given by evaluating maps at markings, and sending maps to their targets and sources: 
	\begin{align*}
		\ev \colon &\Mbar_{g,n}(\CX,\underline{\mu}) \rightarrow \CX^n  
		& \rho \colon \Mbar_{g,n}(\CX,\underline{\mu}) \rightarrow \Mbar^\circ_{h,m}\\
		\pi\colon& \Mbar_{g,n}(\CX,\underline{\mu}) \rightarrow \Mbar_{g,n+\sum \ell(\mu^i)},  
	\end{align*}
	where $\ell(\mu)$ is the length of a partition $\mu$. Let 
	\[ 
	\underline{\gamma}:= \gamma_1\boxtimes \dots \boxtimes\gamma_n \in H^*(\CX^n), \quad \alpha \in H^*(\Mbar^\circ_{h,m})
	\]
	be cohomology classes on $\CX^n$ and $\Mbar^\circ_{h,m}$, respectively. Lastly, we denote
	\[ 
	\Aut(\underline{\mu}):=\prod \Aut(\mu^i),
	\]
	where $\Aut(\mu)$ is the automorphism group of a partition $\mu$, the same notation will be used for the cardinality of this group.  
	\begin{definition}
		We define a Gromov--Witten cycle in $H^*(\Mbar_{g,n+\sum\ell(\mu^i)})$, 
		\begin{equation*}
			\mathsf{GW}_{g,\underline{\mu}}(\underline{\gamma};\alpha) := \frac{1}{\Aut(\underline{\mu})} \cdot  \pi_*([\Mbar_{g,n}(\CX,\underline{\mu})]^{\mathrm{vir}} \cap \ev^*(\underline{\gamma})\cap \rho^*(\alpha )  ).
		\end{equation*}
	\end{definition}
	One can view insertions $\rho^*(\alpha)$ as restrictions to some loci $B \subset \Mbar^\circ_{h,m}$, e.g., a closed point in $\Mbar^\circ_{h,m}$. In this case, it might be more appropriate to replace $\rho^*(\alpha)$ by the fiber product with respect to  $B \rightarrow  \Mbar^\circ_{h,m}$. Consequently, the symbol $\alpha$ in the notation of cycles above can also be regarded as a place-holder for a fiber product. 
	\subsection{Hurwitz cycle}	Similarly, for moduli spaces of admissible covers, there exist the following natural morphisms given by evaluating covers at  additional $k$ branching points associated to $\underline{\eta}$, and sending covers to their targets and  sources: 
	\begin{align*}
		\tilde{\ev}  \colon & \overline{H}(\CX,\underline{\mu},\underline{\eta}) \rightarrow \CX^k  
		&\rho \colon \overline{H}(\CX,\underline{\mu},\underline{\eta}) \rightarrow \Mbar^\circ_{h,m}\\
		\pi\colon& \overline{H}(\CX,\underline{\mu},\underline{\eta}) \rightarrow \Mbar_{g_0, \sum \ell(\mu^i)+\sum \ell(\eta^i)},	  
	\end{align*}
	where $g_0$ is determined by the Riemann--Hurwitz formula applied to $\underline{\mu}$, $\underline{\eta}$ and $h$.
	Let 
	\[ 
	\underline{\gamma}:= \gamma_1\boxtimes \dots \boxtimes\gamma_k \in H^*(\CX^k), \quad \alpha \in H^*(\Mbar^\circ_{h,m})
	\]
	be cohomology classes on $\CX^k$ and $\Mbar^\circ_{h,m}$, respectively.  
	\begin{definition}
		We define a Hurwitz cycle in $H^*(\Mbar_{g_0, \sum \ell(\mu^i)+\sum \ell(\eta^i)})$, 
		\begin{equation*}
			\mathsf{H}_{\underline{\mu}, \underline{\eta}}(\underline{\gamma};\alpha):= \frac{1}{\Aut(\underline{\mu})\Aut(\underline{\eta})}\cdot  \pi_*([\overline{H}(\CX,\underline{\mu},\underline{\eta})] \cap \tilde{\ev}^*(\underline{\gamma})\cap \rho^*(\alpha) ).
		\end{equation*}
			\end{definition}
		
		As for Gromov--Witten cycles, the symbol $\alpha$ can  be treated as a place-holder for a fiber product  with respect to some $B \rightarrow \overline{H}(\CX,\underline{\mu},\underline{\eta})$. %Both kinds of cycles are set to zero in unstable cases. 

		%More generally, in both cases we can also introduce intersections at the markings on the target corresponding to ramifications $\underline{\mu}$, which we will avoid for the notational symplicity

	%\begin{remark} For algebraic insertions $\gamma_i$, our lift to the level of Chow groups. However, working with Chow groups excludes non-algebraic insertions (e.g., those from $H^\mathrm{odd}(X)$).
	%\end{remark}
	\subsection{Star-shaped graphs} \label{Star} Let $\Gamma$ be a star-shaped rooted graph  with leaves,  genus labels on  vertices and partition labels on edges, as depicted in Figure \ref{tree1}. Leaves are labelled by the elements of the set $\{1,\dots , n\}$. We order non-root vertices by a standard order - the order that agrees with the order of the genus labels. Let $N_i\subseteq \{1,\dots,n\}$ be the labels of leaves at $i$-th vertex,  and $n_i=|N_i|$. The set of labels on leaves $N_i$ can be empty.

	We assume that our graphs satisfy the following numerical conditions: 
	\begin{equation} \label{stability}
		|\eta^i|=d, \quad 2g_i-2+d+\ell(\eta^{i})+n_i>0 \quad \text{ for } i=1, \dots, k.
	\end{equation}
	For brevity, we will refer to labelled star-shaped graphs with leaves simply as graphs. The automorphism group  $\Aut(\Gamma)$ of a graph is defined to be the automorphism group of the set of triples $\{(g_1,N_1, \eta_1), \dots, (g_k,N_k, \eta_k) \}$, the same notation will be used for the cardinality of the group.    	 %\footnote{We do not require an automorphism to preserve the order of vertices but only the labels. }
	
	\begin{figure}[!ht]
		\centering
		\begin{tikzpicture}		
			
			\node (1) at (0,0) {};
			\node (2) at (-1.7,-2) {};
			\node (3) at (-0.2,-2) {};
			\node (5) at (0.8, -2) {\dots};
			\node (6) at (1.7, -2) {};
			
			\draw[black, thick] (0,0)--(-1.7,-2);
			\filldraw[ fill=white,draw=white] (-0.85,-1) circle (.25cm);
			\node at (-0.85,-1) {$\eta_1$};
			
			\draw[black, thick] (0,0)--(-0.2,-2);
			\filldraw[ fill=white,draw=white] (-0.05,-1) circle (.25cm);
			\node at (-0.05,-1) {$\eta_2$};
			
			\draw[black,thick] (0,0)--(1.7,-2);
			\filldraw[ fill=white,draw=white] (0.9,-1) circle (.25cm);
			\node at (0.9,-1) {$\eta_k$};
			
			\draw[black, thick] (-1.7,-2)--(-1.7,-2.6);
			\draw[black, thick] (-1.7,-2)--(-2,-2.6);
			\draw[black, thick] (-1.7,-2)--(-1.4,-2.6);
			\draw[black, thick] (-0.2,-2)--(-0.2,-2.6);
			\draw[black, thick] (1.7,-2)--(1.4,-2.6);
			\draw[black, thick] (1.7,-2)--(2,-2.6);
			
			\filldraw[thick, fill = white] (0,0) circle (.25cm) node at (1) {$g_0$};
			\filldraw[thick, fill=white] (-1.7,-2) circle (.25cm) node at (2) {$g_1$};
			\filldraw[thick, fill = white] (-0.2,-2) circle (.25cm) node at (3) {$g_2$};
			\filldraw[thick, fill = white] (1.7,-2) circle (.25cm) node at (6) {$g_k$};
			
			\node (7) at (-1.7,-2.8) {$\underbrace{           }$};
			\node (8) at (-1.7,-3.2) {$N_1$};
			
			\node (7) at (-0.2,-2.8) {$\underbrace{    }$};
			\node (8) at (-0.2,-3.2) {$N_2$};
			
			\node (7) at (1.7,-2.8) {$\underbrace{    }$};
			\node (8) at (1.7,-3.2) {$N_k$};
			
		\end{tikzpicture}
		\caption{Star-shaped graph.} \label{tree1}
	\end{figure}
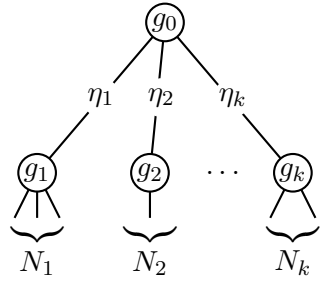

	To each graph $\Gamma$, and a choice of $(g,n)$ with $\underline{\mu}=(\mu^1, \dots,\mu^m)$, we associate
	\[
	M_\Gamma=\Mbar_{g_0,\sum\ell(\mu^i)+\sum \ell(\eta^i)}\times \prod^{k}_{i=1}\Mbar_{g_i, n_i+\ell(\eta^i)},
	\]
	together with a gluing map 
	\[
	\mathrm{gl}_\Gamma \colon M_\Gamma \rightarrow \Mbar_{g,\sum \ell(\mu^i)+n},
	\]
	such that $g=\sum^{k}_{i=1}(g_i+\ell(\eta^{i}))+ g_0-k$, and  markings are ordered (labelled) by the parts of partitions with a standard order and by labels of the leaves; curves are glued along markings with the same label. On unstable components $\Mbar^\circ_{0,1}$ and $\Mbar^\circ_{0,2}$, the gluing map $\mathrm{gl}_\Gamma$ corresponds to the forgetful map and to the map which glues two marked points, respectively. 
	
	On the root component $\Mbar_{g_0,\sum \ell(\mu^i)+\sum \ell(\eta^i)}$, we define 
	\[
	\tilde{\psi}_i:=\eta^i_1\psi_{\{\eta^i_1\}},
	\]
	where $\psi_{\{\eta^i_1\}}$ is the $\psi$-class corresponding to the point labelled by the part $\{\eta^i_1\}$ of the partition $\eta^i=(\eta^i_1,\dots, \eta^i_{\ell(\eta^i)})$. Note that over the Hurwitz loci, we have  
	\[
	\eta^i_j\psi_{\{\eta^i_j\}}=\eta^i_{j'}\psi_{\{\eta^i_{j'}\}},
	\]
	because $\eta^i_j\psi_{\{\eta^i_j\}}$ is equal to the $\psi$-class of the target. Hence for the statement of Theorem \ref{maintheorem}, the insertion of $\tilde{\psi}_i$ is independent of the choice of the part of the partition.
	
	Finally,  to a graph $\Gamma$ and a vector of classes $(\gamma_1,\dots, \gamma_n)$, $\gamma_i \in H^*(\CX)$,  we  associate a class 
	\begin{equation} \label{classGamma}
		\gamma_\Gamma:= \bigg(\prod_{j \in N_1  } \gamma_j \bigg) \boxtimes \dots \boxtimes \bigg(\prod_{j \in N_k} \gamma_j \bigg) \in H^*(\CX^k),
	\end{equation}
	where $N_i$ is the set of leaves  at the $i$-th vertex. 
	
	\subsection{Vertex}
	
	Consider a parametrized projective line $\p^1$ with a $\BC^*$-action, 
	\[ t \cdot [x,y]:=[t\cdot x,y], \quad t\in \BC^*,\]
	we set $0:=[0,1]$ and $\infty:=[1,0]$. 
	\begin{definition} Let $V_{g,n,\eta} \subset \Mbar_{g,n}(\p^1/\infty, \eta)$
		be the space of stable maps from possibly disconnected curves of genus $g$ with $n$ markings to $\p^1$, which are admissible\footnote{Note that maps to expanded degenerations of $\p^1$ are excluded in $V_{g,n,\eta}$.} over $\infty\in \p^1$ with a ramification profile $\eta$. Ramification points over $\infty$ have a standard order. Let $V^\circ_{g,n,\eta}$ be  the similarly defined space but with connected sources. 
	\end{definition}
	The virtual dimension of $V_{g,n,\eta}$ can be readily determined,
	\[
	\mathrm{vdim}(V_{g,n,\eta})=2g-2+d+\ell(\eta)+n.
	\] The space $V_{g,n,\eta}$ also inherits the $\BC^*$-action from $\p^1$, 
	\[ \BC^* \curvearrowright V_{g,n,\eta}.\]
	The $\BC^*$-fixed locus $V_{g,n,\eta}^{\BC^*}$ is proper. In fact, any element of $V_{g,n,\eta}^{\BC^*}$ consists of a contracted genus $g$ curve over $0\in\BP^1$ which is connected to $\ell(\eta)$ many $\BP^1$-tubes mapping to the target $\p^1$ with degree $\eta_i$. As a result, we have a map
	\[ \pi\colon V_{g,n,\eta}^{\BC^*} \rightarrow \overline{M}_{g,n+\ell(\eta)} \]
	of degree $\frac{1}{\prod_{j}\eta_j}$, which coincides with the forgetful map. 	Let $N^{\mathrm{vir}}$ be the virtual normal complex of $V_{g,n,\eta}^{\BC^*}$. Finally, we denote 
	\[z:=e_{\BC^*}(\mathbf{z})\]
	for the weight-1 representation $\mathbf{z}$ of $\BC^*$.   
	
	\begin{definition} \label{defnI}
		We define $I$-functions,
		\[
		\mathsf{I}_{g,n,\eta}(z) := z \cdot \prod^{\ell(\eta)}_{j=1} \eta_j \cdot \pi_*\left( \frac{[V_{g,n,\eta}^{\BC^*}]}{e_{\BC^*}(N^{\mathrm{vir}})} \right) \in H^*(\Mbar_{g, n+\ell(\eta)})[z^\pm], 
		\]	
		such that localised equivariant classes are  expanded in the range $|z| >1$. 
	\end{definition}	
	
	\subsection{Hodge classes}	\label{Hodgeclasses}
	
	Taking the disjoint union of stable maps gives us an isomorphism
	\begin{equation} \label{decomposition1}
		\coprod_{\{(g_i,N_i,\eta^i)\}_{i=1}^l} \prod_{i=1}^l V^\circ_{g_i,N_i,\eta^i}= V_{g,n,\eta},
	\end{equation}
	where the disjoint union runs over partitions $\{(g_i,N_i,\eta^i)\}_{i=1}^l$ satisfying
	\[\sum_{i=1}^l g_i-l+1 = g, \  \coprod_{i=1}^l N_i = \{1,\ldots,n\}, \ \eta^i\neq \emptyset \text{ and } \coprod_{i=1}^l \eta^i = \eta.\]
	We have an identically looking decompositon for the $\BC^*$-fixed loci. As a result one can rewrite the $I$-function in the following way:
	\begin{equation} \label{decomposition}
		\mathsf{I}_{g,n,\eta}(z)= z \cdot  \sum_{\{(g_i,N_i,\eta^i)\}_{i=1}^l}\prod_{i=1}^l \mathsf{I}^\circ_{g_i,n_i,\eta^i}(z),
	\end{equation}
	where 
	\begin{align*}
		\mathsf{I}^\circ_{g, n, \eta }(z):=  \prod^{\ell(\eta)}_{j=1} \eta_j \cdot \pi_*\left( \frac{[(V^\circ_{g,n,\eta})^{\BC^*}]}{e_{\BC^*}(N^{\mathrm{vir}})} \right) \in H^*(\Mbar^\circ_{g, n+\ell(\eta)})[z^\pm], 
	\end{align*}
	which admits an expression in terms of $\lambda$- and $\psi$-classes, obtained by analysing $e_{\BC^*}(N^{\mathrm{vir}})$. 
	
	\begin{proposition}[\cite{GP},\cite{OPeq}] \label{localisation_formula}
		For $(g,n)\neq (0,0),(0,1)$, we have 
		\[ 
		\mathsf{I}^\circ_{g,n,\eta}(z)=z^{\ell(\eta)-|\eta|-1}\cdot \prod^{\ell(\eta)}_{j=1}\frac{\eta_j^{\eta_j}}{\eta_j!} \cdot \frac{\Lambda^\vee(z)}{\prod^{\ell(\eta)}_{j=1} \frac{z}{\eta_j}-\psi_j} \in H^*(\Mbar^\circ_{g, n+\ell(\eta)})[z^\pm],
		\]
		where $\Lambda^\vee(z):= e(\mathbb{E}^\vee \mathbf{z})=\sum (-1)^j \lambda_j z^{g-j}$, and  $ \mathbb{E}$ is the Hodge bundle. For $(g,n)=(0,0),(0,1)$, 
		\begin{align*}
		\mathsf{I}^\circ_{0,n,(\eta_1)}(z)&=z^{1-\eta_1-n} \cdot \frac{\eta_1^{\eta_1-1+n}}{\eta_1!}, \quad  
		\mathsf{I}^\circ_{0,0,(\eta_1,\eta_2)}(z)=z^{-\eta_1-\eta_2} \cdot \frac{\eta_1^{\eta_1+1}\eta_2^{\eta_2+1}}{\eta_1!\eta_2!(\eta_1+\eta_2)} \in \BQ[z^\pm]. 
		\end{align*}

	\end{proposition}
	
	By introducing certain conventions on what $\psi$-classes are for unstable values of $(g,n)$ as it was done in \cite{OPeq}, the formulas for $\mathsf{I}^\circ_{g,n,\eta}(z)$ admit a consistent expression for all values of  $(g,n)$.
	
	\subsection{Wall-crossing formula}

	\begin{theorem} \label{maintheorem} Assume $(h,m)\neq (0,1)$. Then we have
		\[ 
		\mathsf{GW}_{g,\underline{\mu}}(\underline{\gamma};\alpha)=\sum_{\Gamma} \frac{1}{\Aut(\Gamma)} \cdot \mathrm{gl}_{\Gamma*}\left(\mathsf{H}_{\underline{\mu},\underline{\eta}}(\gamma_\Gamma ;\alpha) \boxtimes \prod^{k}_{i=1} \mathsf{I}_{g_i,n_i,\eta^i}(- \tilde{\psi_i})\right), 
		\]
		such that the sum is taken over all graphs $\Gamma$ with $n$ leaves,   subject to the following condition:
		\begin{align*}
			g&=\sum^{k}_{i=1}(g_i+\ell(\eta^{i}))+ g_0-k, \\
			2g_0-2&= d(2h-2)+\sum^m_{i=1} (d-\ell(\mu^i))+\sum^{k}_{i=1} (d-\ell(\eta^i)).
		\end{align*}
		The negative powers of $\psi$-classes are set to zero. 
	\end{theorem}
	
	Observe that the classes $\mathsf{H}_{\underline{\mu},\underline{\eta}}(\gamma_\Gamma ;\alpha)$ live on the root component of the graph $\Gamma$, while the classes $\mathsf{I}_{g,n,\eta}(z)$ are on the non-root components. However, the substitution of $z$ by a $\psi$-class is done on the root component. For example, for a class $A\boxtimes Bz^k$, where $A$ is on the root component and $B$ is on the non-root components, the substitution of $z$ by a $\psi$-class takes the following form:
	\[
	A\boxtimes Bz^k \mapsto A(-\tilde{\psi})^k\boxtimes B.
	\]
	The formula also admits a natural geometric interpretation:
	\[
	[\textsf{Stable maps}]= \sum[\textsf{Ramified covers}]\cdot [\textsf{Contracted components}],
	\] 
	where the sum is taken over all ways of attaching contracted components to a cover. In particular, these components may be attached either away from or at the ramification locus. In this perspective, $I$-functions are precisely the (virtual) classes associated to moduli spaces of contracted components of curves. The case $(h,m)=(0,1)$ admits a different wall-crossing  formula, which is considered in Section \ref{SectionELSV}.
	
	\section{Proof of Theorem}
	\subsection{Idea of the proof} The result follows from the localisation on the master space associated to $\epsilon$-unramification, defined in full generality in \cite{NuG} and considered in a slightly different context in \cite{N22}.\footnote{The set-up of \cite{N22} did not include markings in the source curves, and only relative insertions were permitted.} In fact, Theorem \ref{maintheorem} is a cycle-valued upgrade of the numerical wall-crossing discussed in \cite{NuG}. The proof is quite similar, the difference is that we project our classes to $\Mbar_{g,n}$ instead of a point. We establish a wall-crossing formula for one wall in Section \ref{allwalls}.  Theorem \ref{maintheorem} is obtained by applying (\ref{onewall}) to all walls along the path from stable maps to admissible covers.
	\subsection{Stability}	Let  $(C,\mathbf{p})$ and $(X, \mathbf{x})$ be two marked nodal curves.  Consider a map, 
	\[
	f\colon (C,\mathbf{p})  \rightarrow  ( X, \mathbf{x}),
	\]
	admissible over $\mathbf{x}$ and nodes. Let $\mathbf{p}'\subset C$ be the set of points given by the preimages of marked points $\mathbf{x} \subset X$. To such a map, we can associate a branching complex, as considered in \cite{FP} in the case of a smooth target, 
	\[ 
	Rf_*\Big[f^*\Omega^{\log}_{( X,\mathbf{x})} \rightarrow \Omega^{\log}_{(C,\mathbf{p}, \mathbf{p}')}\Big] \in \mathrm{D}^{b}( X),
	\]
	where  $\Omega^{\log}_{(C, \mathbf{p},\mathbf{p}')}$ is the logarithmic cotangent bundle associated to a marked nodal curve, the same apples to $\Omega^{\log}_{(X,\mathbf{x})}$. The branching complex is a zero-dimensional perfect complex supported over:
	\begin{itemize}
		\item  branching points of $f$ away from $\mathbf{x}$ and nodes, 
		\item images of marked points $\mathbf{p}$,
		\item images of contracted components and nodes of $C$.
	\end{itemize}

	By taking the support of the complex weighted by the Euler characteristics, we obtain the branching divisor of $f$, 
	\[
	\mathrm{br}(f) \in \mathrm{Div}(X).
	\]
	We use the branching divisor to define  $\epsilon$-unramification of maps. 
	\begin{definition} \label{defnstability}
		Given a number $\epsilon \in \BR_{>0}$. An admissible map $f\colon (C,\mathbf{p}) \rightarrow  (X, \mathbf{x})$ is $\epsilon$-unramified, if
		\begin{itemize}
			\item for all points $x \in X$, $\mathrm{mult}_x(\mathrm{br}(f))\leq 1/\epsilon$,
			\item for all rational tails $P \subset (X, \mathbf{x})$, $\deg(\mathrm{br}(f)_{|P})> 1/\epsilon$,
			\item $\Aut(f)$ is finite. 
		\end{itemize}
	\end{definition}
	Observe that for $\epsilon \ll1$, an $\epsilon$-unramified map is a stable map. For $\epsilon=1$, it is an admissible cover with simple ramifications and markings in the source. For $\epsilon>1$, it is an unramified admissible cover without marked points in the source. 
	
	Let 
	\[\Mbar^{\epsilon}_{g,n}(\CX, \underline{\mu}) \colon  (Sch/\BC)^\circ \rightarrow Grpd\]
	be the moduli space of $\epsilon$-unramified maps of degree $d$ from possibly disconnected curves of genus $g$ with $n$ marked points, and with ramifications over $m$ marked points on the target  specified a vector of partitions $\underline{\mu}:=(\mu^1,\dots,\mu^m)$. These are proper spaces by \cite[Theorem 3.15]{NuG}; see also  \cite{N22}, where the case of a moving curve without markings on source curves was considered. 
	
	By the discussion above, these spaces specialise to $\Mbar_{g,n}(\CX,\underline{\mu})$ for $\epsilon \ll1$,  and  to $\overline{H}(\CX,\underline{\mu})$ for $\epsilon >1$ and $n=0$.  They also admit both kinds of evaluation morphisms, 
	\begin{align*}
		&\ev \colon \Mbar^{\epsilon}_{g,n}(\CX, \underline{\mu}) \rightarrow \CX^n \quad  \\
		&\tilde{\ev} \colon \Mbar^{\epsilon}_{g,n}(\CX, \underline{\mu}) \rightarrow \CX^m,
	\end{align*}
	given by marked points on the the source and the branching points on the target, respectively.

	\subsection{Master space}
	Let $\epsilon_0=1/d_0$ for some $d_0 \in \BN$. We call such $\epsilon_0$ a wall. Let $\epsilon_-$ and $\epsilon_+$ be values close to a wall from the left and the right, respectively.  For each wall $\epsilon_0$, there exists a proper master space equipped with a perfect obstruction theory 
	\[ 
	M\Mbar^{\epsilon_0}_{g,n}(\CX,\underline{\mu}),
	\] 
	constructed in \cite[Section 6]{NuG} and also in \cite[Section 3]{N22}. The construction is inspired by a similar master space introduced by Zhou in the context of GIT quasimaps  \cite{YZ}, using the theory of entangled tails. The master space  carries a $\BC^*$-action, 
	\[ 
	\BC^* \curvearrowright M\Mbar^{\epsilon_0}_{g,n}(\CX,\underline{\mu}). 
	\]
	The $\BC^*$-fixed locus admits the following expression, 
	\begin{align} \label{fixed_locus}
		M\Mbar^{\epsilon_0}_{g,n}(\CX,\underline{\mu})^{\BC^*}&=\Mbar^{\epsilon_-}_{g,n}(\CX, \underline{\mu}) \sqcup \widetilde{M}^{\epsilon_+}_{g,n}(\CX, \underline{\mu}) \sqcup \coprod_{w(\Gamma)=d_0} F_{\Gamma}.
	\end{align}
	The first two components in (\ref{fixed_locus}) are divisors, and, in particular, their normal bundles inside $M\Mbar^{\epsilon_0}_{g,n}(\CX,\underline{\mu})$ are of rank 1; the spaces $\widetilde{M}^{\epsilon_+}_{g,n}(\CX, \underline{\mu})$ are certain blow-ups of $\Mbar^{\epsilon_+}_{g,n}(\CX, \underline{\mu})$, see Section \cite[Section 6.7]{NuG}; the pushforwards of their virtual fundamental classes agree with those of $\Mbar^{\epsilon_+}_{g,n}(\CX, \underline{\mu})$. 
	
	On the other hand, as explained in \cite[Section 6.7]{NuG}, elements of $F_{\Gamma}$ are maps to curves with rational tails of degree $d_0$ with respect to the branching divisor, such that the shape of a map (i.e., how the genus and ramifications are redistributed) is determined by the graph $\Gamma$; over the distinguished degree-$d_0$ rational tails specified by non-root vertices of the graph, the map belongs to $V^{\BC^*}_{g_i, n_i, \eta^{i}}$.  The union is taken over all star-shaped graphs $\Gamma$ of weight $d_0$, that is 
	\[
	2g_i-2+ d+\ell(\eta^{i})+n_i=d_0  \quad\text{ for } i=1, \dots, k,
	\]
	the weight of a graph is denoted by $w(\Gamma)$. We allow leaves at the root vertex, labelling them by a possibly empty subset $N_0 \subseteq \{1,\dots ,n\}$, as depicted in Figure \ref{tree2}. We also require the graph to satisfy the condition $g=\sum^{k}_{i=1}(g_i+\ell(\eta^{i}))+ g_0-k$.  
	
	\begin{figure}[!ht]
		\centering
		\begin{tikzpicture}		
			
			\node (1) at (0,0) {};
			\node (2) at (-1.7,-2) {};
			\node (3) at (-0.2,-2) {};
			\node (5) at (0.8, -2) {\dots};
			\node (6) at (1.7, -2) {};

			\draw[black, thick] (0,0)--(-1.7,-2);
			\filldraw[ fill=white,draw=white] (-0.85,-1) circle (.25cm);
			\node at (-0.9,-1) {$\eta_{1}$};
			
			\draw[black, thick] (0,0)--(-0.2,-2);
			\filldraw[ fill=white,draw=white] (-0.05,-1) circle (.25cm);
			\node at (0.05,-1) {$\eta_{2}$};
			
			\draw[black,thick] (0,0)--(1.7,-2);
			\filldraw[ fill=white,draw=white] (0.9,-1) circle (.25cm);
			\node at (1,-1) {$\eta_{k}$};
			
			\draw[black, thick] (-1.7,-2)--(-1.7,-2.6);
			\draw[black, thick] (-1.7,-2)--(-2,-2.6);
			\draw[black, thick] (-1.7,-2)--(-1.4,-2.6);
			\draw[black, thick] (-0.2,-2)--(-0.2,-2.6);
			\draw[black, thick] (1.7,-2)--(1.4,-2.6);
			\draw[black, thick] (1.7,-2)--(2,-2.6);
			\draw[black, thick] (0,0)--(0.3,0.6);
			\draw[black, thick] (0,0)--(0.0,0.6);
			\draw[black, thick] (0,0)--(-0.3,0.6);
			
			\filldraw[thick, fill = white] (0,0) circle (.25cm) node at (1) {$g_0$};
			\filldraw[thick, fill=white] (-1.7,-2) circle (.25cm) node at (2) {$g_1$};
			\filldraw[thick, fill = white] (-0.2,-2) circle (.25cm) node at (3) {$g_2$};
			\filldraw[thick, fill = white] (1.7,-2) circle (.25cm) node at (6) {$g_k$};
			
			\node (7) at (0,0.75) {$\overbrace{}$};
			\node (8) at (0,1.1) {$N_0$};
		\end{tikzpicture}
		\caption{Star-shaped graph with leaves at the root.} \label{tree2} 
	\end{figure}
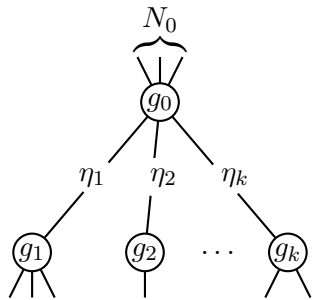

	Let $F^\mathrm{ord}_{\Gamma}$ be the space associated to $F_\Gamma$ by putting a standard order (i.e., the order which is in agreement with genus labels of $\Gamma$) on rational tails  of degree $d_0$ and on the ramification points over nodes. Then the space $F^\mathrm{ord}_{\Gamma}$ is a $\BZ_k$-gerbe over the following product, 
	\[ 
	F^\mathrm{ord}_{\Gamma} \rightarrow \widetilde{M}^{\epsilon_+}_{g_0,n_0}(\CX, \underline{\mu}, \underline{\eta})\times \prod^{k}_{i=1} V^{\BC^*}_{g_i, n_i, \eta^{i}}. 
	\]

	\subsection{Localisation on the master space} \label{Sectionloc} The master space admits a forgetful morphism, 
	\[ 
	\pi \colon M\Mbar^{\epsilon_0}_{g,n}(\CX,\underline{\mu}) \rightarrow \Mbar_{g,n}(\CX, \underline{\mu}),
	\]
	given by applying stabilisation to all rational tails of target curves, composing maps with the stabilisation morphism, and stabilising the source curves. We could go even further down to $\Mbar_{g,n+\sum \ell(\mu^i)}$, but pullbacks of $\psi$-classes from $\Mbar_{g,n+\sum \ell(\mu^i)}$ and relative $\psi$-classes on $M\Mbar^{\epsilon_0}_{g,n}(\CX,\underline{\mu})$ do not agree in general due to ramifications away from the marked points on the target. These are retained in $\Mbar_{g,n}(\CX, \underline{\mu})$ but forgotten in $\Mbar_{g,n+\sum \ell(\mu^i)}$. Nevertheless, they do agree in the case of admissible covers, which will allow us to consider the forgetful morphism to $\Mbar_{g,n+\sum\ell(\mu^i)}$ in the final chamber of the wall-crossing. 
	
		\begin{remark}   Using forgetful morphisms to $\Mbar_{g,n+\sum \ell(\mu^i)}$ for all walls is also possible. In this case, one needs to insert $\psi$-classes before applying pushforwards. This, however, will require a more clumsy notation in the considerations that follow. 
	\end{remark}
	
	By applying the localisation formula \cite{GP} and pushing forward via $\pi$, we obtain the following relation in  $H^*(\Mbar_{g,n}(\CX, \underline{\mu}))[z^\pm]$, 
	\begin{multline*}
		[M\Mbar^{\epsilon_0}_{g,n}(\CX,\underline{\mu})]^{\mathrm{vir}}\cap A \\ 
		= \pi_*\left( \frac{[\Mbar^{\epsilon_-}_{g,n}(\CX,  \underline{\mu})]^{\mathrm{vir}}\cap A}{-z+\mathrm{c}_1(L_-)}\right)+\pi_*\left(\frac{[\Mbar^{\epsilon_+}_{g,n}(\CX,  \underline{\mu})]^{\mathrm{vir}}\cap A}{z+\mathrm{c}_1(L_+)}\right)+\sum_{w(\Gamma)=d_0} \pi_*\left(\frac{[F_{\Gamma}]^{\mathrm{vir}}\cap A}{e_{\BC^*}(N^\mathrm{vir})}\right), 
	\end{multline*}
	such that we expand rational functions in $z$ in the range $|z| >1$. The class $A$ is an insertion associated to markings on the source and branching points on the target, 
	\[ 
	A:=\ev^*(\gamma_1\boxtimes \dots \boxtimes \gamma_n) \cdot  \tilde{\ev}^*(\tilde{\gamma}_1 \boxtimes  \dots \boxtimes \tilde{\gamma}_m )\cdot \rho^*(\alpha).
	\]
	Since $M\Mbar^{\epsilon_0}_{g,n}(\CX,\underline{\mu})$ is proper by \cite[Theorem 6.13]{NuG}, the left-hand side of the relation above does not involve negative powers of $z$. After taking the residue at $z=0$, i.e., the coefficients of $\frac{1}{z}$, we therefore obtain 
	\[ 
	\pi_* \left([\Mbar^{\epsilon_-}_{g,n}(\CX,  \underline{\mu})]^{\mathrm{vir}}\cap A \right)= \pi_*\left( [\Mbar^{\epsilon_+}_{g,n}(\CX,  \underline{\mu})]^{\mathrm{vir}}\cap A \right)+\sum_{w(\Gamma)=d_0} \mathrm{Res}_{z} \pi_*\left(\frac{[F_{\Gamma}]^{\mathrm{vir}}\cap A}{e_{\BC^*}(N^\mathrm{vir})}\right),
	\]
	it remains to analyse the residue of the wall-crossing components $F_{\Gamma}$; note that since we are taking the residue, a precise form of lines bundles $L_-$ and $L_+$ is not required. 
	\subsection{Residue} By the analysis of the obstruction theory of the fixed components from \cite[Proposition 6.19]{NuG}, we have
	\begin{multline} \label{equation1}
		\mathrm{Res}_z\pi_*\left(\frac{[F_{\Gamma}]^{\mathrm{vir}}\cap A}{e_{\BC^*}(N^\mathrm{vir})}\right)=\\
		\mathrm{Res}_z\frac{1}{\Aut(\Gamma)}\cdot\frac{1}{ \Aut(\underline{\eta})}\cdot \mathrm{gl}_{\Gamma*}\left(\pi_*\left(\frac{[\widetilde{M}^{\epsilon_+}_{g_0,n_0}(X, \underline{\mu}, \underline{\eta})]^{\mathrm{vir}}\cap A_\Gamma)}{-z-\sum^{\infty}_{i=0}D_i} \right) \boxtimes \prod^{k}_{i=1} \mathsf{I}_{g_i,n_i,\eta^{i}}(z- \tilde{\psi_i})\right), \end{multline}
	where $D_i$ are certain divisor classes on $\widetilde{M}^{\epsilon_+}_{g_0,n_0}(X, \underline{\mu}, \underline{\eta})$, which are defined in \cite[Defintion 6.2]{NuG}; we will not specify what they are but instead state the result of their contribution to the residue in (\ref{res}). The class $A_\Gamma$ is derived from $A$, such that some insertions via markings on the source become insertions via branching points on the target, depending on whether these markings are on the root or a non-root vertex in a graph $\Gamma$,
	\begin{equation*} 
		A_\Gamma=  \ev^*( \boxtimes_{i\in N_0}\gamma_i) \cdot \tilde{\ev}^*(\tilde{\gamma}_1 \boxtimes  \dots \boxtimes  \tilde{\gamma}_m)\cdot \tilde{\ev}'^*(\gamma_\Gamma)\cdot \rho^*(\alpha),
	\end{equation*}
	where $\tilde{\ev}'$ is the evaluation map associated to the branching points with the ramification profile $\underline{\eta}$; the class $\gamma_\Gamma$ is defined as in (\ref{classGamma}).  The factor $\Aut(\Gamma)$ is due to the ordering of rational tails and ramification points in $F_\Gamma$. Lastly, the classes $\tilde{\psi}_i$ in the expression above are associated to nodes of target curves.  On moduli spaces of maps with relative points they admit an identification with $\psi$-classes on the source via the formula involving a part of the partition  given in the end of Section \ref{Star}.

	For the purpose of obtaining simpler combinatorial factors in the next formula, it is more convenient to put arbitrary orders on non-root vertices of graphs. This amounts to summing over all ordered graphs with factorial factors. We will denote a graph with an arbitrary order by $\Gamma_{\mathrm{ord}}$.  Hence using the analysis presented in the proof of  \cite[Theorem 7.3.3]{YZ}, we obtain that 
	\begin{multline} \label{res}
		\sum_{w(\Gamma)=d_0} \mathrm{Res}_{z} \pi_*\left(\frac{[F_{\Gamma}]^{\mathrm{vir}}\cap A}{e_{\BC^*}(N^\mathrm{vir})}\right) = 
		\sum_{w(\Gamma_\mathrm{ord})=d_0} \sum^{k-1}_{r=0} \sum_{\underline{b}} \frac{(-1)^r}{r!(k-r)!}\cdot \frac{1}{\Aut(\underline{\eta})}\cdot \\ \mathrm{gl}_{\Gamma_\mathrm{ord}*}\bigg(\pi_*\left( [\Mbar^{\epsilon_+}_{g_0,n_0}(\CX, \underline{\mu}, \underline{\eta})]^{\mathrm{vir}} \cap  A_{\Gamma_{\mathrm{ord}}} \right)
		\boxtimes   \prod^k_{i=1} [\mathsf{I}_{g_i,n_i,\eta^{i}}(z- \tilde{\psi_i})] _{b_\ell}\bigg), 
	\end{multline}
	where  we sum over ordered graphs $\Gamma_{\mathrm{ord}}$  of weight $d_0$ with $k$ non-root vertices; $\underline{b}$ runs through $k$-tuples of integers such that $b_1 + \hdots+ b_k=0$ and $b_{k-r+1},\hdots, b_k<0$; $[\mathsf{I}_{g_i,n_i,\eta^i}(z- \tilde{\psi_i})] _{b_\ell}$ is the coefficient of $\mathsf{I}_{g_i,n_i,\eta^i}(z- \tilde{\psi_i}) $ at $z^{b_\ell}$. 
	
		\begin{remark} To obtain (\ref{res}) from (\ref{equation1}), it is also necessary to consider the wall-crossing for $(h,m)=(0,1)$, which is presented in the proof of Theorem \ref{ELSV}. The reason is that the contributions of $D_i$ are expressible in terms of polar parts of $I$-functions, as explained in \cite[Lemma 7.2.1]{YZ}. 
	\end{remark}

	By noticing that the terms with $b_\ell<0$ cancel out due to the sign (see the end of the proof of \cite[Theorem 7.3.3]{YZ}), we obtain that (\ref{res}) simplifies to 
	\begin{equation} \label{onewall0}
		\sum_{w(\Gamma_{\mathrm{ord}})=d_0}  \frac{1}{k!}\cdot \frac{1}{ \Aut(\underline{\eta})}\cdot \mathrm{gl}_{\Gamma_\mathrm{ord}*}\bigg(\pi_*\left( [\Mbar^{\epsilon_+}_{g_0,n_0}(\CX, \underline{\mu}, \underline{\eta})]^{\mathrm{vir}} \cap  A_{\Gamma_{\mathrm{ord}}} \right)  \boxtimes   \prod^k_{i=1} [\mathsf{I}_{g_i,n_i,\eta^{i}}(z- \tilde{\psi_i})] _{0}\bigg),
	\end{equation}
	such that
	\[[\mathsf{I}_{g_i,n_i,\eta^{i}}(z- \tilde{\psi_i})] _{0}=\mathsf{I}_{g_i,n_i,\eta^{i}}(-\tilde{\psi_i}), \]
	where negative powers of $\psi$-classes are set to be zero.

	\subsection{Crossing all walls} \label{allwalls}  By removing the order on graphs in (\ref{onewall0}), we obtain the following relation in $H^*(\Mbar_{g,n}(\CX, \underline{\mu}))$:
	\begin{multline} \label{onewall}
		\pi_* \left([\Mbar^{\epsilon_-}_{g,n}(\CX,  \underline{\mu})]^{\mathrm{vir}}\cap A \right)= \pi_*\left( [\Mbar^{\epsilon_+}_{g,n}(\CX,  \underline{\mu})]^{\mathrm{vir}}\cap A \right) \\
		+ \sum_{w(\Gamma)=d_0}\  \frac{1}{\Aut(\Gamma)}\cdot\frac{1}{\Aut(\underline{\eta})}\cdot \mathrm{gl}_{\Gamma*} \bigg(\pi_*\left( [\Mbar^{\epsilon_+}_{g_0,n_0}(\CX, \underline{\mu}, \underline{\eta})]^{\mathrm{vir}} \cap  A_\Gamma \right)  \boxtimes   \prod^k_{i=1} \mathsf{I}_{g_i,n_i,\eta^{i}}(- \tilde{\psi_i})\bigg).
	\end{multline}
	We apply this formula to all walls on the way from $\epsilon \ll1$ to $\epsilon >1$. In particular, we inductively express the wall-crossing summands in the formula using the same formula for the next wall. The result is Theorem \ref{maintheorem} but on $\Mbar_{g,n}(\CX, \underline{\mu})$. 
	
	To obtain the statement of Theorem  \ref{maintheorem} on $\Mbar_{g,n+\sum \ell(\mu^i)}$, we pushforward classes via the forgetful morphism $\Mbar_{g,n}(\CX, \underline{\mu}) \rightarrow \Mbar_{g,n+\sum \ell(\mu^i)}$. Over loci of admissible covers,  this no longer involves stabilisation of curves, since source curves of admissible covers together with ramification points are stable. Hence restrictions of $\psi$-classes to Hurwitz cycles agree on both spaces.  In the process of crossing all walls, we obtain graphs without any markings on the root vertex and whose root genus $g_0$ is determined by the partition labels on the edges. 
	%We can make the formula slightly more explicit by applying Lemma \ref{empty}. For short we denote the partition $(2,1^{d-2})$ by $(2)$. It will also reveal a very improtant feature of the wall-crossing formula, namely, the summation over simple ramifications. 

	%\begin{corollary}\label{cor: points_kill_each_other} We have
	%	\begin{multline*}
		%	\langle \tau_{k_1}(\gamma_1) \dots \tau_{k_n}(\gamma_n)  \mid \widetilde{\tau}_{k'_1}(\gamma'_1)   \dots \widetilde{\tau}_{k'_m}(\gamma'_m) \rangle^-_{\underline{\eta}}\\
		%	=\sum_{\vv{g}} \Big\langle \emptyset \ \Big| \ \ \widetilde{\tau}_{k'_1}(\gamma'_1)   \dots \widetilde{\tau}_{k'_m}(\gamma'_m) \cdot  \prod^{k}_{i=1} I_{g_i,\eta'^{i}}\Big( \prod_{j\in N_i} \tau_{k_j}(\gamma_j) \Big), \mathbbm{1}^b \Big\rangle^{+}_{\underline{\eta}, \underline{\eta}', (2)^b}\frac{1}{k!b!},
		%	\end{multline*}
	%	where we sum over integers $b\in \BZ_{>0}$ and partitions $\vv{g}=((g_1,N_1, \eta'^{1}),\dots,(g_k,N_k,\eta'^{k}))$ for which $N_i\neq \emptyset$.
	%\end{corollary}

	\section{Degree one maps} \label{section_one}
	\subsection{Degree one maps}For $d=1,n=0$, the wall-crossing formula takes a particularly simple form, as the only contributing $I$-function is $\mathsf{I}_{g_i, (1)}(z)$. 
	\begin{corollary} \label{cor_one} Assume $d=1$, $n=0$. Then
		\[ 
		\mathsf{GW}_{g}=\sum_{(g_1,\dots,g_k)} \frac{1}{\Aut(g_1,\dots, g_k)}\cdot \mathrm{gl}_{*}\left( [\Mbar^\circ_{h,k}] \boxtimes \prod^{k}_{i=1} \mathsf{I}_{g_i, (1)}(- \tilde{\psi_i})\right), 
		\]
		such that $g_i>0$, $g=h+ \sum^k_{i=1} g_i$ and 
		\[ 
		\mathsf{I}_{g_i, (1)}(z)=\frac{\Lambda^\vee(z)}{z-\psi_1} \in H^*(\Mbar^\circ_{g_i, 1})[z^\pm]. 
		\] 
	\end{corollary}
	
	\begin{proof}
		This follows directly from Theorem \ref{maintheorem} and the observation that the space of admissible covers of degree one to $\CX \rightarrow \Mbar^\circ_{h,0}$ is isomorphic to $\Mbar^\circ_{h,0}$. Graphs with labels $(1)$ on edges correspond to $k$-tuples of integers. By Proposition \ref{localisation_formula}, we obtain an expression for $I$-functions. 
	\end{proof}
	
	Moduli spaces of degree one maps to the moving elliptic curve were considered in \cite{COP} in connection to the moduli space of abelian varieties $\CA_g$ and the Torelli map $\mathrm{Tor} \colon \CM^{\mathrm{ct}}_g \rightarrow \CA_g$. The fiber product of $\CA_1\times \CA_{g-1}$ with respect to the Torelli map is isomorphic to the moduli space of stable maps of degree one to the moving elliptic curve, such that the pullback class and the virtual fundamental class match.  Hence Corollary \ref{cor_one} provides a more effective method for computing $\mathrm{Tor}^*[\CA_1\times \CA_{g-1}]$ than the one used in \cite[Section 6]{COP}. We examine the case of the moving elliptic curve in greater detail below.  

\subsection{Genus three}  Let $d=1$, $(g,n)=(3,0)$ and $(h,m)=(1,0)$, i.e., we are in the situation of the moving elliptic curve without markings.  Figure \ref{tree4} depicts graphs involved in this wall-crossing.

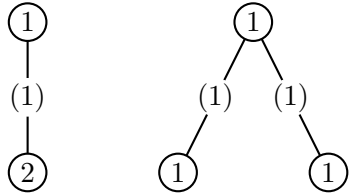
\begin{figure}[!ht]
	\centering
	\begin{tikzpicture}		
	
		\node (1) at (-3,0) {};
		\node (2) at (-3,-2) {}; 
		
		\draw[black, thick] (-3,0)--(-3,-2);
		
		\filldraw[thick, fill = white] (-3,0) circle (.25cm) node at (1) {$1$};
		\filldraw[thick, fill=white] (-3,-2) circle (.25cm) node at (2) {$2$};
		
		\filldraw[ fill=white,draw=white] (-3,-1) circle (.25cm);
		\node at (-3,-1) {$(1)$};
		
		\node (1) at (0,0) {};
		\node (2) at (-1,-2) {}; 
		\node (3) at (1,-2) {}; 
		
		\draw[black, thick] (0,0)--(-1,-2);
		\draw[black, thick] (0,0)--(1,-2);
		
		\filldraw[thick, fill = white] (0,0) circle (.25cm) node at (1) {$1$};
		\filldraw[thick, fill=white] (-1,-2) circle (.25cm) node at (2) {$1$};
		\filldraw[thick, fill=white] (1,-2) circle (.25cm) node at (3) {$1$};
		
		\filldraw[ fill=white,draw=white] (-0.5,-1) circle (.25cm);
		\node at (-0.5,-1) {$(1)$};
		
		\filldraw[ fill=white,draw=white] (0.5,-1) circle (.25cm);
		\node at (0.5,-1) {$(1)$};
	\end{tikzpicture}\caption{Graphs for $d=1$, $(g,n)=(3,0)$.} \label{tree4} 
\end{figure}	

There are just two  $I$-functions in this case: 
\begin{align} \label{g3}
	\begin{split}
		&\hspace{-3cm} \mathsf{I}_{1,(1)}(z)=\frac{z-\lambda_1}{z-\psi_1}=\mathbbm{1}+\mathrm{O}(z^{-1})\\
		&\hspace{-3cm} \mathsf{I}_{2, (1)}(z)= \frac{z^2-\lambda_1 z+ \lambda_2}{z-\psi_1 }=z -\lambda_1+ \psi_1+ \mathrm{O}(z^{-1}),
	\end{split}
\end{align}
 \noindent By Corollary \ref{cor_one}, 
\[
\mathsf{GW}_3=-[\tilde{\psi}_1,\mathbbm{1}  ]_{1,2}-[\mathbbm{1}, \lambda_1]_{1,2}+[\mathbbm{1}, \psi_1]_{1,2}+\frac{1}{2}[\mathbbm{1}, \mathbbm{1},\mathbbm{1}]_{1^3},
\]	
where $[ \alpha_1, \dots, \alpha_k ]_{g_0, g_1,\hdots,g_k}=\mathrm{gl}_* (\alpha_0 \boxtimes \alpha_1 \boxtimes \dots\boxtimes \alpha_k$). Over the locus of compact-type curves $M^\mathrm{ct}_{3}\subset \Mbar_{3}$, using \texttt{admcycles}  \cite{admcycles}, the expression above yields
\[ 
\mathsf{GW}_{3|M^\mathrm{ct}_{3}}=24\lambda_2,
 \]
 in agreement with \cite{COP}. 
 
 \subsection{Genus four} Let $(g,n)=(4,0)$.   Figure \ref{tree5} depicts graphs involved in this wall-crossing. 

 \begin{figure}[H] 
 	\centering 
 	\begin{tikzpicture}		
 		
 		\node (1) at (-3,0) {};
 		\node (2) at (-3,-2) {}; 
 		
 		\draw[black, thick] (-3,0)--(-3,-2);
 		
 		\filldraw[thick, fill = white] (-3,0) circle (.25cm) node at (1) {$1$};
 		\filldraw[thick, fill=white] (-3,-2) circle (.25cm) node at (2) {$3$};
 		
 		\filldraw[ fill=white,draw=white] (-3,-1) circle (.25cm);
 		\node at (-3,-1) {$(1)$};
 		
 		\node (1) at (0,0) {};
 		\node (2) at (-1,-2) {}; 
 		\node (3) at (1,-2) {}; 
 		
 		\draw[black, thick] (0,0)--(-1,-2);
 		\draw[black, thick] (0,0)--(1,-2);
 		
 		\filldraw[thick, fill = white] (0,0) circle (.25cm) node at (1) {$1$};
 		\filldraw[thick, fill=white] (-1,-2) circle (.25cm) node at (2) {$1$};
 		\filldraw[thick, fill=white] (1,-2) circle (.25cm) node at (3) {$2$};
 		
 		\filldraw[ fill=white,draw=white] (-0.5,-1) circle (.25cm);
 		\node at (-0.5,-1) {$(1)$};

 		\filldraw[ fill=white,draw=white] (0.5,-1) circle (.25cm);
 		\node at (0.5,-1) {$(1)$};
 		
 		%%%
 		
 		\node (1) at (4,0) {};
 		\node (2) at (4,-2) {}; 
 		\node (3) at (3,-2) {}; 
 		\node (4) at (5,-2) {}; 
 		
 		\draw[black, thick] (4,0)--(4,-2);
 		\draw[black, thick] (4,0)--(3,-2);
 		\draw[black, thick] (4,0)--(5,-2);
 		
 		\filldraw[thick, fill = white] (4,0) circle (.25cm) node at (1) {$1$};
 		\filldraw[thick, fill=white] (4,-2) circle (.25cm) node at (2) {$1$};
 		\filldraw[thick, fill=white] (3,-2) circle (.25cm) node at (3) {$1$};
 		\filldraw[thick, fill=white] (5,-2) circle (.25cm) node at (4) {$1$};
 		
 		\filldraw[ fill=white,draw=white] (4,-1) circle (.25cm);
 		\node at (4,-1) {$(1)$};
 		
 		\filldraw[ fill=white,draw=white] (3.4,-1) circle (.25cm);
 		\node at (3.4,-1) {$(1)$};
 		
 		\filldraw[ fill=white,draw=white] (4.6,-1) circle (.25cm);
 		\node at (4.6,-1) {$(1)$};
 	\end{tikzpicture}\caption{Graphs for $d=1$, $(g,n)=(4,0)$.} \label{tree5} 
 \end{figure}
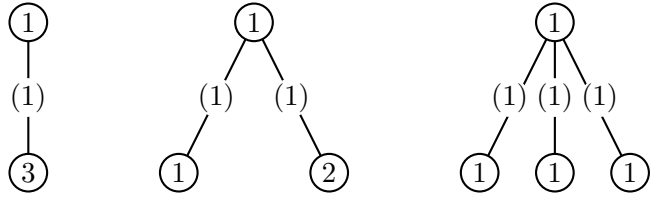

In addition to (\ref{g3}), there is one more $I$-function:
 \begin{align*}	
 	&\hspace{1cm}\mathsf{I}_{3, (1)}(z)= \frac{z^3-\lambda_1 z^2+ \lambda_2z-\lambda_3}{z-\psi_1}=z^2 -\lambda_1z+\lambda_2+ \psi_1z -\lambda_1 \psi_1+ \psi_1^2+\mathrm{O}(z^{-1}). 
 \end{align*}

\noindent By Corollary \ref{cor_one}, 
\begin{align*}
\mathsf{GW}_4&=[\tilde{\psi}_1^2,\mathbbm{1}]_{1,3}+[\tilde{\psi}_1,\lambda_1]_{1,3}+[\mathbbm{1},\lambda_2]_{1,3}- [\tilde{\psi}_1,\psi_1]_{1,3}+[\mathbbm{1},\psi_1\lambda_1]_{1,3}++[\mathbbm{1},\psi_1^2]_{1,3} \\
&-[\tilde{\psi}_1,\mathbbm{1},\mathbbm{1} ]_{1^2,2}-[\mathbbm{1},\mathbbm{1}, \lambda_1]_{1^2,2}+[\mathbbm{1},\mathbbm{1}, \psi_1]_{1^2,2}+\frac{1}{6}[\mathbbm{1}, \mathbbm{1},\mathbbm{1}, \mathbbm{1}]_{1^4}.
\end{align*}
Over the locus of compact-type curves $M^\mathrm{ct}_{4}\subset \Mbar_{4}$, using  \texttt{admcycles} \cite{admcycles}, the expression above yields  
\[ 
\mathsf{GW}_{4|M^\mathrm{ct}_{4}}=20\lambda_3,
\]
also in agreement with \cite{COP}. 
Observe that,  in contrast  to \cite{COP}, we sum only over star-shaped graphs, and the wall-crossing formula holds over the entire $\Mbar_g$. 
	\section{Degree two maps}
	\subsection{$I$-functions in degree two} For $d=2,n=0$, there will be  two types of $I$-functions contributing to the wall-crossing formula: 
	\begin{itemize}
		\item $\mathsf{I}_{g, (2)}(z)$, which corresponds to attaching a connected curve of genus $g$ to a simple ramification of admissible covers.
		\item $\mathsf{I}_{g, (1^2)}(z)$, which corresponds to attaching a possibly disconnected curve of genus $g$ to conjugate points of admissible covers.\footnote{Points in a fiber of a degree two admissible cover away from branching points.}
	\end{itemize}
	Recall that in (\ref{decomposition}) we allow  $\Mbar^\circ_{0,1}$, $\Mbar^\circ_{0,2}$.    In particular, a special role is played by the $I$-function for  
	$(g,\eta)=(0,(2))$, 
	\[
	\mathsf{I}_{0,(2)}(z)=\mathbbm{1},
	\]
	which when $n=0$ corresponds just to a simple ramification on admissible covers without any components attached to it. 
	
	\subsection{Relative projective line} Let us focus on maps to $\p^1$ relative to two points $\{0,\infty\} \subset \p^1$ up to  $\BC^*$-scaling that fixes these points. This corresponds to $(h,m)=(0,2)$, i.e., the universal curve over $\Mbar^\circ_{0,2}$. We fix the ramification profiles over $\{0,\infty\}$ to be $(2)$. On the Gromov--Witten side, we have moduli spaces
	\[\Mbar_{g, (2)^2}(\p^1/\{0,\infty\})^\sim,  \]
	whose associated cycles are Double ramification cycles \cite{JPPZ}, 
	\[ 
	\mathsf{GW}_{g,(2)^2}=\mathsf{DR}_g(2,-2) \in H^*(\Mbar_{g,2}), 
	\]
such that simple ramifications over $\{0,\infty\}$ are marked.
	
	On the Hurwitz side, we have moduli spaces of admissible covers with ramification profiles given by partitions $(1,1)$ and $(2)$,
	\[
	\overline{H}_{(1,1)^{v}, (2)^{w+2}}(\p^1/\{0,\infty\})^\sim.
	\] 
The associated cycles are cycles of loci of Hyperellipic curves, such that all Weierstrass points and some conjugate hyperelliptic points are marked,  
	\[
	\mathsf{H}_{(1,1)^{v}, (2)^{w+2}}=\mathsf{Hyp}_{g_0,w+2, 2v }\in H^*(\Mbar_{g_0, w+2+2v}), 
	\]
	where $g_0=\frac{w}{2}$. 
	
	Let us also introduce a shortened notation for Hyperelliptic cycles without markings on Weierstrass points away from $\{0,\infty\}$, which in the language of $I$-functions corresponds to inserting $\mathsf{I}_{0,(2)}=\mathsf{I}_{0,(2)}(z)=\mathbbm{1}$ at all simple ramifications, 
	\begin{align*}
		\mathsf{Hyp}_{g}&:=\frac{1}{(2g)!}\cdot  \mathrm{gl}_{\Gamma*}\left(\mathsf{Hyp}_{g,2g+2,0}\boxtimes \prod^{2g}_{i=1} \mathsf{I}_{0,(2)}\right) \\
		&\hspace{0.1cm}=\frac{1}{(2g)!} \cdot \mathrm{fg}_*(\mathsf{Hyp}_{g,2g+2,0}) \in H^*(\Mbar_{g_0,2}),
	\end{align*}
	where  $\Gamma$ is the graph with the label $g$ at the root vertex, $0$ at non-root vertices and  $2g$ edges whose labels are $(2)$, and $\mathrm{fg} \colon \Mbar_{g,2g+2} \rightarrow \Mbar_{g,2}$ is the forgetful map.
	
	The wall-crossing  relates Double ramification and Hyperelliptic cycles, providing a recursive formula which expresses $\mathsf{Hyp}_{g}$ in terms of $\mathsf{DR}_g(2,-2)$, $ \mathsf{I}_{g_i, \eta^i}(z)$ and $\mathsf{Hyp}_{g_0,w+2, 2v }$ for $g_0< g$. Note that the attachment of $I$-functions at simple ramifications takes place away from $\{0,\infty\}$.
\pagebreak
	
	\begin{corollary} \label{DoubleHyp}  We have
		\begin{multline*}
			\mathsf{Hyp}_{g}
			=\mathsf{DR}_g(2,-2)\\- \sum_{\Gamma} \frac{1}{\Aut(\Gamma)} \cdot \mathrm{gl}_{\Gamma*}\left(\mathsf{Hyp}_{g_0,w+2, 2v } \boxtimes \prod^{v}_{i=1}  \mathsf{I}_{g_i, (1,1)}(-\tilde{\psi}_i) \boxtimes \prod^{v+w}_{i=v+1}  \mathsf{I}_{g_i, (2)}(-\tilde{\psi}_i) \right),
		\end{multline*}
		such that sum is taken over graphs subject to the following condition:
		\begin{align*}
			g=\sum^{v+w}_{i=1}(g_i+\ell(\eta^{i}))+ g_0-v-w, \quad g_0=\frac{w}{2},
		\end{align*} 
		where $v$ is the number of edges of the graph whose label is $(1,1)$, while $w$ is the number of edges whose label is $(2)$. We exclude the graph with $2g$ edges labelled by $(2)$, as the corresponding term was moved to the left-hand side of the equation. 
	\end{corollary}
	
	%By the results of \cite{EH}, the marked Hyperelliptic cycles can be determined from the unmarked ones. A formula for Double ramification cycles was derived in \cite{JPPZ}. 
	
	\subsection{Genus three} Let $d=2, (g,n)=(3,0)$. Using Proposition \ref{localisation_formula} and (\ref{decomposition}), we can easily determine the relevant $I$-functions. Notice that $I$-functions associated to the partition $(1^2):=(1,1)$ have several contributions depending on different connected components of $V_{g,n,\eta}$ which appear in the decomposition (\ref{decomposition1}).

Figure \ref{tree3} depicts the graphs involved in this wall-crossing. For simplicity, we do not include vertices associated to $\mathsf{I}_{0, (2)}(z)$, as their number can be deduced from other labels. We also omit those graphs whose $I$-functions are purely polar, as they do not contribute to the wall-crossing formula. 

	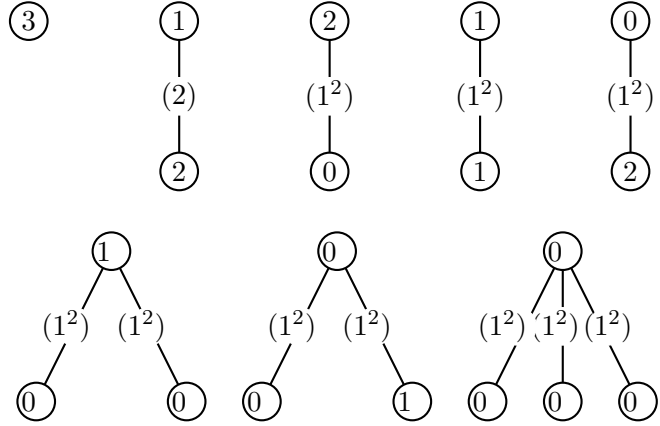
\begin{figure}[!ht]
	\centering
	\begin{tikzpicture}		
		
		\node (1) at (0,0) {};
		\filldraw[thick, fill = white] (0,0) circle (.25cm) node at (1) {$3$};

		\node (1) at (4,0) {};
		\node (2) at (4,-2) {}; 
		
		\draw[black, thick] (4,0)--(4,-2);
		
		\filldraw[thick, fill = white] (4,0) circle (.25cm) node at (1) {$2$};
		\filldraw[thick, fill=white] (4,-2) circle (.25cm) node at (2) {$0$};
		
		\filldraw[ fill=white,draw=white] (4,-1) circle (.25cm);
		\node at (4,-1) {$(1^2)$};

		\node (1) at (2,0) {};
		\node (2) at (2,-2) {}; 
		
		\draw[black, thick] (2,0)--(2,-2);
		
		\filldraw[thick, fill = white] (2,0) circle (.25cm) node at (1) {$1$};
		\filldraw[thick, fill=white] (2,-2) circle (.25cm) node at (2) {$2$};
		
		\filldraw[ fill=white,draw=white] (2,-1) circle (.25cm);
		\node at (2,-1) {$(2)$};

		\node (1) at (8,0) {};
		\node (2) at (8,-2) {}; 
		
		\draw[black, thick] (8,0)--(8,-2);
		
		\filldraw[thick, fill = white] (8,0) circle (.25cm) node at (1) {$0$};
		\filldraw[thick, fill=white] (8,-2) circle (.25cm) node at (2) {$2$};
		
		\filldraw[ fill=white,draw=white] (8,-1) circle (.25cm);
		\node at (8,-1) {$(1^2)$};

		\node (1) at (6,0) {};
		\node (2) at (6,-2) {}; 
		
		\draw[black, thick] (6,0)--(6,-2);
		
		\filldraw[thick, fill = white] (6,0) circle (.25cm) node at (1) {$1$};
		\filldraw[thick, fill=white] (6,-2) circle (.25cm) node at (2) {$1$};
		
		\filldraw[ fill=white,draw=white] (6,-1) circle (.25cm);
		\node at (6,-1) {$(1^2)$};
		
	\end{tikzpicture}
	
	\vspace{0.5cm} 
	\hspace{-0.3cm} 
	\begin{tikzpicture}	
		%\node (1) at (0,0) {};
		%\node (2) at (-1,-2) {}; 
		%\node (3) at (1,-2) {}; 
		
		%\draw[black, thick] (0,0)--(-1,-2);
		%\draw[black, thick] (0,0)--(1,-2);
		
		%\filldraw[thick, fill = white] (0,0) circle (.25cm) node at (1) {$0$};
		%\filldraw[thick, fill=white] (-1,-2) circle (.25cm) node at (2) {$2$};
		%\filldraw[thick, fill=white] (1,-2) circle (.25cm) node at (3) {$0$};
		
		%\filldraw[ fill=white,draw=white] (-0.5,-1) circle (.25cm);
		%\node at (-0.5,-1) {$(2)$};
		
		%\filldraw[ fill=white,draw=white] (0.5,-1) circle (.25cm);
		%\node at (0.5,-1) {$(1^2)$};
		
		%%%%
		
		\node (1) at (3,0) {};
		\node (2) at (2,-2) {}; 
		\node (3) at (4,-2) {}; 
		
		\draw[black, thick] (3,0)--(2,-2);
		\draw[black, thick] (3,0)--(4,-2);
		
		\filldraw[thick, fill = white] (3,0) circle (.25cm) node at (1) {$1$};
		\filldraw[thick, fill=white] (2,-2) circle (.25cm) node at (2) {$0$};
		\filldraw[thick, fill=white] (4,-2) circle (.25cm) node at (3) {$0$};
		
		\filldraw[ fill=white,draw=white] (2.5,-1) circle (.25cm);
		\node at (2.5,-1) {$(1^2)$};
		
		\filldraw[ fill=white,draw=white] (3.5,-1) circle (.25cm);
		\node at (3.5,-1) {$(1^2)$};
		
		%%%%
		
		\node (1) at (6,0) {};
		\node (2) at (5,-2) {}; 
		\node (3) at (7,-2) {}; 
		
		\draw[black, thick] (6,0)--(5,-2);
		\draw[black, thick] (6,0)--(7,-2);
		
		\filldraw[thick, fill = white] (6,0) circle (.25cm) node at (1) {$0$};
		\filldraw[thick, fill=white] (5,-2) circle (.25cm) node at (2) {$0$};
		\filldraw[thick, fill=white] (7,-2) circle (.25cm) node at (3) {$1$};
		
		\filldraw[ fill=white,draw=white] (5.5,-1) circle (.25cm);
		\node at (5.5,-1) {$(1^2)$};
		
		\filldraw[ fill=white,draw=white] (6.5,-1) circle (.25cm);
		\node at (6.5,-1) {$(1^2)$};
		%%%%
		
		\node (1) at (9,0) {};
		\node (2) at (9,-2) {}; 
		\node (3) at (8,-2) {}; 
		\node (4) at (10,-2) {}; 
		
		\draw[black, thick] (9,0)--(9,-2);
		\draw[black, thick] (9,0)--(8,-2);
		\draw[black, thick] (9,0)--(10,-2);
		
		\filldraw[thick, fill = white] (9,0) circle (.25cm) node at (1) {$0$};
		\filldraw[thick, fill=white] (9,-2) circle (.25cm) node at (2) {$0$};
		\filldraw[thick, fill=white] (8,-2) circle (.25cm) node at (3) {$0$};
		\filldraw[thick, fill=white] (10,-2) circle (.25cm) node at (4) {$0$};
		
		\filldraw[ fill=white,draw=white] (9,-1) circle (.25cm);
		\node at (9,-1) {$(1^2)$};
		
		\filldraw[ fill=white,draw=white] (8.4,-1) circle (.25cm);
		\node at (8.3,-1) {$(1^2)$};
		
		\filldraw[ fill=white,draw=white] (9.6,-1) circle (.25cm);
		\node at (9.7,-1) {$(1^2)$};
		
	\end{tikzpicture}

	\caption{Graphs for $d=2$, $(g,n)=(3,0)$.} \label{tree3} 
	
\end{figure}

For the partition $(2)$, $I$-functions are: 
\begin{align*}	
	&\hspace{-3cm} \mathsf{I}_{0,(2)}(z)=\mathbbm{1}\\
	&\hspace{-3cm} \mathsf{I}_{1, (2)}(z)= 2z^{-1}  \cdot \frac{z-\lambda_1}{\frac{z}{2}-\psi_1}= \mathrm{O}(z^{-1})\\
	& \hspace{-3cm} \mathsf{I}_{2, (2)}(z)= 2z^{-1}  \cdot \frac{z^2-\lambda_1 z+ \lambda_2}{\frac{z}{2}-\psi_1}= 4 \mathbbm{1}+ \mathrm{O}(z^{-1}).
\end{align*}
\noindent For the partition $(1^2)$, they are: 	
\begin{align*}		
		&\mathsf{I}_{0, (1^2)}(z)_{|\Mbar^\circ_{0,2}}=\frac{1}{2z}=\mathrm{O}(z^{-1})\\
		&\mathsf{I}_{0, (1^2)}(z)_{|\Mbar^\circ_{1,1}\times \Mbar^\circ_{0,1}}=   \frac{z-\lambda_1}{z-\psi_1}=\mathbbm{1}+\mathrm{O}(z^{-1})\\
		&\mathsf{I}_{1, (1^2)}(z)_{| \Mbar^\circ_{1,2} }= \frac{z-\lambda_1}{(z-\psi_1)(z-\psi_2)}=\mathrm{O}(z^{-1})\\
		&\mathsf{I}_{1, (1^2)}(z)_{| \Mbar^\circ_{1,1} \times \Mbar^\circ_{1,1} }=   \frac{z-\lambda_1}{z-\psi_1}\cdot \frac{z-\lambda_1}{z(z-\psi_1)}= \mathrm{O}(z^{-1})\\
		&\mathsf{I}_{1, (1^2)}(z)_{| \Mbar^\circ_{2,1} \times \Mbar^\circ_{0,1} }=   \frac{z^2-\lambda_1 z+ \lambda_2}{z-\psi_1 }=z -\lambda_1+ \psi_1 + \mathrm{O}(z^{-1})\\
		&\mathsf{I}_{2, (1^2)}(z)_{| \Mbar^\circ_{2,2} }=  \frac{ z^2-\lambda_1 z+ \lambda_2 }{(z-\psi_1)(z-\psi_2)}=\mathbbm{1}+\mathrm{O}(z^{-1}) \\
		&\mathsf{I}_{2, (1^2)}(z)_{| \Mbar^\circ_{2,1} \times   \Mbar^\circ_{1,1} }= \frac{z^2-\lambda_1 z+ \lambda_2}{z-\psi_1 }\cdot \frac{z-\lambda_1}{z(z-\psi_2)}=\mathbbm{1}+\mathrm{O}(z^{-1})\\
		&\mathsf{I}_{2, (1^2)}(z)_{| \Mbar^\circ_{3,1}\times \Mbar^\circ_{0,1} }= \frac{z^3-z^2\lambda_1+z\lambda_2-\lambda_3}{(z-\psi_1)}\\
		&\hspace{3.1cm}= z^2-z\lambda+\lambda_2+z\psi_1-\lambda_1\psi_1+\psi_1^2+\mathrm{O}(z^{-1}).
	\end{align*}

As explained in Appendix \ref{AppendixJohannes}, using  \texttt{admcycles} \cite{admcycles},  Corollary \ref{DoubleHyp} was verified  for $g\leq 3$.  For computations, it is important to include all vertices of graphs, even those corresponding to $\mathsf{I}_{0, (2)}(z)$, because of the insertion of $\psi$-classes on root components.\footnote{A $\psi$-class and the pullback of $\psi$-class by a forgetful map differ by a boundary divisor.}  The same applies to conjugate points - we can forget a conjugate point which is not glued to any other components only after inserting $\psi$-classes.

	\section{ELSV formula} \label{SectionELSV}
	
	\subsection{The polar part of $I$-function}  Theorem \ref{maintheorem} involves only the non-polar part of $\mathsf{I}_{g,n,\mu}(z)$ because negative powers $\psi$-classes are set to zero after the substitution of $z$.  The case of $\Mbar^\circ_{0,1}$  is special - there exists a different wall-crossing formula involving  the polar part of $ \mathsf{I}_{g,n,\eta}(z)$. Let 
	\[
	\overline{H}_{\mu, \underline{\eta}}(\p^1/\infty)^\sim
	\]	
	be the space of admissible covers to $(\p^1,\infty)$ with a ramification $\mu$ at $\infty \in \p^1$ up to automorphisms $\Aut(\p^1,\infty)$. In other words, this is a moduli space of admissible covers of the universal curve over $\Mbar^\circ_{0,1}$. Let 
	\[
	\widetilde{\mathsf{H}}_{\mu, \underline{\eta}} \in H^*(\Mbar_{g_0, \ell(\mu)+\sum \ell(\eta^i)})
	\] be the associated cycle. By $\tilde{\psi}_\infty$ we denote the $\psi$-class associated to $\infty \in \p^1$. Recall that  $\Fz(\mu)=\Aut(\mu)\prod^{\ell(\mu)}_{j=1}\mu_j$. 
	
	\begin{theorem} \label{ELSV} We have 
		\[[\mathsf{I}_{g,n,\mu}(z)]_{z^{<0}}= \Fz(\mu)  \cdot  \sum_{\Gamma} \frac{1}{\Aut(\Gamma)}\cdot \mathrm{gl}_{\Gamma*}\left(\frac{\widetilde{\mathsf{H}}_{\mu, \underline{\eta}}}{z-\tilde{\psi}_\infty} \boxtimes \prod^{k}_{i=1} \mathsf{I}_{g_i,n_i,\eta^i}(- \tilde{\psi_i})\right), \]
		such that the sum is taken over all graphs $\Gamma$ with $n$ leaves,  subject to the following condition:
		\begin{align*}
			g=\sum^{k}_{i=1}(g_i+\ell(\eta^{i})) + g_0-k, \quad 2g_0-2= -d-\ell(\mu)+\sum^{k}_{i=1} (d-\ell(\eta^i)).
		\end{align*}
		Negative powers of $\psi$-classes are set to zero.
	\end{theorem}
	
	\begin{proof} Let $\epsilon_+\in \BR_{>0}$ be such that 
		\begin{equation} \label{epsilonplus}
			2g-2+d+\ell(\mu)+n<\epsilon_+<  2g-2+d+\ell(\mu)+n+1.
		\end{equation}
		Note that on the left from the wall $\epsilon_0=2g-2+d+\ell(\mu)+n$, the conditions of $\epsilon$-unramification give an empty moduli space by Definition \ref{defnstability}. Moreover, the master space for the wall $\epsilon_0$ takes a different and simpler form. It is defined as follows. 
		
		Let $\mathfrak{M}_{g,n}(\p^1/\infty, \mu)^\sim$ the moduli stack of maps to $\p^1$ with the ramification profile $\mu$ over $\infty \in \p^1$ up to $\Aut(\p^1,\infty)$, such that the automorphisms of maps to the fixed $\p^1$ are finite (i.e., we require that the automorphisms of maps to $\p^1$ are finite, but they might be infinite if  automorphisms  of the target $\Aut(\p^1,\infty)$ are taken into account). Let $L$ be the line bundle given by the cotangent space at $\infty \in \p^1$.  Consider  the projective bundle, 
		\[
		\p_{\mathfrak{M}_{g,n}(\p^1/\infty, \mu)^\sim}(L\oplus \CO).
		\]
		The $B$-points of this space are tuples $(f,N, v_1, v_2)$, where $f \in \mathfrak{M}_{g,n}(\p^1/\infty, \mu)^\sim(B)$, $N$ is a line bundle on the base scheme $B$, and $v_1 \in H^0(B, L_B\otimes N)$ and $v_2\in H^0(B,N$) are sections with no common zeros. 
		We define
		\[ M\Mbar_{g,n}^{\epsilon_0}(\p^1/\infty, \mu)^\sim \subset \p_{\mathfrak{M}_{g,n}(\p^1/\infty, \mu)^\sim }(L\oplus \CO), \]
		by imposing two conditions: 
		\begin{enumerate}
			\item if $v_1=0$, then maps are $\epsilon^+$-unramified,
			\item  $v_2\neq 0$.
		\end{enumerate}
		It is easy to see that $M\Mbar_{g,n}^{\epsilon_0}(\p^1/\infty, \mu)^\sim$ is proper. In short, when $v_1$ tends to $0$ over the spectrum of a discrete valuation ring $\Delta$, and the underlying map is not $\epsilon_+$-unramified, we  reparametrize $(\p^1,\infty)$, i.e., multiply the section by $\pi^{-k}$ for some positive $k$, where $\pi$ is a uniformiser of $\Delta$.    When $v_2$ tends to $0$, we multiply the section by  $\pi^{k}$. 
		
		We have a $\BC^*$-action on the master space $M\Mbar_{g,n}^{\epsilon_0}(\p^1/\infty, \mu)^\sim$ given by scaling the line bundle $L$, 
		\[
		t\cdot (f,v_1,v_2):=(f,t\cdot v_1, v_2 ). 
		\]
	The fixed points are either $\epsilon^+$-unramified maps with $v_1= 0$ or maps from  $V_{g,n,\mu}^{\BC^*}$ with $v_1\neq 0$, 
		\[ (M\Mbar_{g,n}^{\epsilon_0}(\p^1/\infty, \mu)^\sim )^{\BC^*}=\Mbar_{g,n}^{\epsilon_+}(\p^1/\infty, \mu)^{\sim} \sqcup V_{g,n,\mu}^{\BC^*}. \]
		By taking higher residues in the localisation formula (i.e., coefficients of $\frac{1}{z^k}$), we obtain the relation, 
		\begin{equation} \label{01case}
		[\mathsf{I}_{g,n,\mu}(z)]_{z^{<0}}=\prod^{\ell(\mu)}_{j=1} \mu_j \cdot \pi_*\left( \frac{[\Mbar_{g,n}^{\epsilon_+}(\p^1/\infty, \mu)^{\sim}]^{\mathrm{vir}}}{z-\tilde{\psi}_{\infty}} \right).
		\end{equation}
		The factor  $\prod^{\ell(\mu)}_{j=1} \mu_j$ comes in because it appears in our definition of $I$-functions due to  the splitting of curves in the localisation formula of the master space from Section \ref{Sectionloc}; in the localisation formula above curves do not split. 
		
		To finish the proof,  we apply the formula  (\ref{onewall}) to the class on the right of (\ref{01case}), crossing all walls from $\epsilon_+$ of the from (\ref{epsilonplus}) to $\epsilon>1$.  
	\end{proof}
	
	\subsection{ELSV formula} For $n=0$, Theorem \ref{ELSV} specialises to the ELSV formula, if we take the coefficient of the minimal power of $z$. In this case,  by the dimension constraint, the only contributing $I$-functions are 
	\[ \mathsf{I}_{0,(2)}(z)=\mathbbm{1},\]
	i.e., those that correspond to simple ramifications. Indeed, the presence of other $I$-functions would drop the dimension of the root component, forcing the corresponding power of $\tilde{\psi}_\infty$ to vanish. From Theorem \ref{ELSV}, we therefore obtain
	\begin{equation} \label{ELSVeq}
		[z^{-m+1}]  \mathsf{I}_{g,\mu}(z) = \frac{\Fz(\mu)}{m! } \cdot  \tilde{\psi}_{\infty}^{m-2}\cdot \widetilde{\mathsf{H}}_{\mu, (2)^{m}} \in H^*(\Mbar_{g, \ell(\mu)} ), 
	\end{equation}
	where $m=2g-2+\ell(\mu)+d$. By the string equation, Lemma \ref{string},
	\[
	\int_{\Mbar_{g,\ell(\mu)+m}}\tilde{\psi}_{\infty}^{m-2}\cdot \widetilde{\mathsf{H}}_{\mu, (2)^{m}}= \mathrm{Hur}^{\p^1}(\mu, (2)^m),
	\]
	where $\mathrm{Hur}^{\p^1}(\mu, (2)^m)$ is a Hurwitz number, whose definition we recall in  Section \ref{sectionHu}. 
	The relation (\ref{ELSVeq}) also holds on all connected components of $\Mbar_{g,\ell(\mu)}$. Using Proposition \ref{localisation_formula}, we therefore obtain 
	\[ 
	\frac{m!}{\Aut(\mu)} \cdot  \prod^{\ell(\mu)}_{j=1}\frac{\mu_j^{\mu_j}}{\mu_j!} \cdot \frac{\Lambda^\vee(1)}{\prod^{\ell(\mu)}_{j=1} (1-\mu_j\psi_j)} =\mathrm{Hur}^{\p^1}(\mu, (2)^m)^\circ, 
	\]
	which is the ELSV formula, proved in \cite{ELSV,GVH}. Theorem \ref{ELSV} can be viewed as its cycle-valued refinement.
	
%	\subsection{Formula for $\lambda_g$} On a different note, if consider the degree one case of Theorem \ref{ELSV} but the highest power of $z$, we obtain a formula for $\lambda_g$. 
%	\begin{corollary} We have 
	%	\[ \lambda_g=\sum_{(g_1,\dots, g_k)}\frac{1}{\Aut(g_1,\dots,g_k)} \cdot \mathrm{gl}_{*} \cdot \left( [\Mbar^\circ_{0,k}] \boxtimes \prod^{k}_{i=1} \mathsf{I}_{g_i, (1)}(- \tilde{\psi_i})\right), \]
	%	such that $g_i>0$,  $k\geq 2$, $\sum g_i=g$  and 
	
	%	\[ 
%		\mathsf{I}_{g_i, (1)} =\frac{\Lambda^\vee(z)}{z-\psi_1} \in H^*(\Mbar_{g_i, 1}). 
%		\] 
	%	\end{corollary}
	
	\section{Numerical wall-crossing formula}
	\subsection{Fixed curve} In this section, we will focus on the numerical wall-crossing formula applied to a fixed smooth curve $X$ without relative points. More explicitly, this corresponds to inserting $\alpha=[X]$, where $X$ is a smooth genus $h$ curve without markings, and intersecting  cycles with sufficiently many  $\psi$-classes to cut them down to numbers.  
	
	\subsection{Numerical invariants}
	
	Given classes $\gamma_i  \in H^*(X)$, we define Gromov--Witten and Hurwitz invariants:
	\begin{align*}
		\langle \tau_{k_1}(\gamma_1) \dots \tau_{k_n}(\gamma_n)  \rangle^{\mathsf{GW}}_{g}&:=\int_{\mathsf{GW}_g(\underline{\gamma}) }\prod^{n}_{i=1} \psi_i^{k_i}, \\
		\langle \tau_{k_1}(\gamma_1) \dots \tau_{k_n}(\gamma_n)  \rangle^{\mathsf{H}}_{\underline{\eta}}&:=\int_{\mathsf{H}_{\underline{\eta}}(\underline{\gamma}) }\prod^{n}_{i=1} \tilde{\psi}_i^{k_i}.
	\end{align*}	
	We similarly define the numerical $I$-functions, 
	\[
	I_{g,\eta}(z; \tau_{k_1}(\gamma_1) \dots \tau_{k_n}(\gamma_n) ):= \prod^{n}_{i=1} \gamma_i \cdot \int_{ \mathsf{I}_{g,n,\eta}(z)} \prod^{n}_{i=1} \psi^{k_i}_i,	
	\]		
	the same notation applies to the connected $I$-functions 	$I^\circ_{g,\eta}(z; \tau_{k_1}(\gamma_1) \dots \tau_{k_n}(\gamma_n) )$. Note that numerical $I$-functions are elements of the cohomology $H^*(X)[z^\pm]$. 
	
	 By the homogeneity of classes under the integral sign, the  numerical $I$-functions are of the form $A z^{k}$, where $A\ \in H^*(X)$. The $z$-degree of a numerical $I$-function can be computed by the virtual dimension of $V_{g,n,\eta}$, 
	\begin{align*}
		\deg_z I_{g,\eta}(z; \tau_{k_1}(\gamma_1) \dots \tau_{k_n}(\gamma_n) )&=\sum_i k_i+1-\mathrm{vdim}(V_{g,n,\eta}) \\
		&=\sum_i k_i+3-2g-d-\ell(\eta)-n.
	\end{align*}
	We  define $I$-functions without the variable $z$ by taking the coefficient at $z^{\deg_z}$, 
	\[  
	I_{g,\eta} ( \tau_{k_1} \dots \tau_{k_n} ):=[z^{\deg_z}]\int_{ \mathsf{I}_{g,n,\eta}(z)} \prod^{n}_{i=1} \psi^{k_i}_i,
	\]
	where $ \tau_{k}$ is a shorthand for  $\tau_{k}(\mathbbm{1})$.

	\begin{remark} Note that if $g(X)\neq0$, then  forgetful maps  to  $\Mbar_{g,n}$ do not require stabilisation of curves. Hence the pullbacks of $\psi$-classes from $\Mbar_{g,n}$ agree with $\psi$-classes on moduli spaces of maps. If $g(X)=0$, then all results of this section also apply to integrals of $\psi$-classes over moduli spaces of maps. 
	\end{remark}
	
	\subsection{Hurwitz numbers} \label{sectionHu}
	For the point class $\omega \in H^2(X)$, the Hurwitz invariant
	\[\langle \tau_0(\omega) \dots \tau_0(\omega) \rangle^{\mathsf{H}}_{\underline{\eta}}\] gives a Hurwitz number of $X$, i.e., counts of ramified covers with the ramification profile $\underline{\eta}$, 
	\[ \mathrm{Hur}(\eta^1, \dots, \eta^n):= \langle \tau_0(\omega) \dots \tau_0(\omega) \rangle^{\mathsf{H}}_{\underline{\eta}}.  \]
	For Section \ref{sectionGWH}, it will be convenient to extend  $ \mathrm{Hur}(\eta^1, \dots, \eta^n)$ multilinearly to the vector space spanned by partitions of $d$, 
	\[ 
	\mathrm{Hur}(-, \dots, -) \colon \CZ(d)^n \rightarrow \BQ, \quad \CZ(d):=\big\{ \sum_{\eta \vdash d} a_\eta (\eta)  \mid a_\eta \in \BQ \big\}.
	\]
	
	\subsection{Fulton--MacPherson invariants}
	A moduli space of admissible covers admits a finite map to a Fulton--MacPherson space \cite{FM} of $X$ by associating branching points on the target to an admissible cover, 
	\begin{equation} \label{HFM}
		\overline{H}_{\underline{\eta}}(X) \rightarrow {FM}_n(X),
	\end{equation}
	the degree of this map is equal to the Hurwitz number $ \mathrm{Hur}(\underline{\eta})$. 
	
	Consider evaluation map, which is just the forgetful map associated to contraction of bubbles, 
	\[
	\ev \colon FM_n(X) \rightarrow X^n. 
	\]
	We define invariants associated to $FM_n(X)$, 
	\[
	\langle \tau_{k_1}(\gamma_1) \dots \tau_{k_n}(\gamma_n)  \rangle^{\mathsf{FM}}:= \int_{[FM_n(X)]} \ev^*(\gamma_1 \boxtimes \dots \boxtimes \gamma_n) \cdot \prod^n_{i=1} \psi^{k_i}_i. 
	\]
	Using the map (\ref{HFM}), we obtain that
	\begin{equation} \label{invHFM}
		\langle \tau_{k_1}(\gamma_1) \dots \tau_{k_n}(\gamma_n)  \rangle^{\mathsf{H}}_{\underline{\eta}}= \mathrm{Hur}(\underline{\eta}) \cdot \langle \tau_{k_1}(\gamma_1) \dots \tau_{k_n}(\gamma_n)  \rangle^{\mathsf{FM}}.
	\end{equation}
This allows to determine integrals on moduli spaces of admissible covers. In particular, we have the string equation, which will be used in Lemma \ref{raw}. 
	\begin{lemma} \label{string}
		The Fulton--MacPherson invariants satisfy the string equation, 
		\[
		\langle \tau_{k_1}(\gamma_1) \dots \tau_{k_n}(\gamma_n) \cdot \mathbbm{1} \rangle^{\mathsf{FM}}= \sum^n_{i=1} \langle \tau_{k_1}(\gamma_1) \dots \tau_{k_i-1}(\gamma_i) \dots \tau_{k_n}(\gamma_n)  \rangle^{\mathsf{FM}}.  
		\]
	\end{lemma}
\begin{proof} The result follows by applying forgetful maps between Fulton--MacPherson spaces, and using the relation between a $\psi$-class and the pullback of a $\psi$-class. 
\end{proof}	
	%\begin{proof} This is a straightforward application of the forgetful map and the expression of the pullback of a $\psi$-class. 
%	\end{proof}
	
	There exists another wall-crossing relating integrals on Fulton--MacPherson spaces to integrals on symmetric products, see \cite[Section 6]{NHilb}. 
	
	\subsection{Hodge integrals}
	Consider the numerical $I$-function $I^\circ_{g,\eta}(z, \prod \tau_{k_i})$. If we replace every occuring cohomology class $\gamma$ by $z^{\deg_{\BC}(\gamma)}\gamma$ and multiply the whole integral by $z^{-\dim(\overline{M})}$, the integral is left unchanged. This results in
	\begin{align*}
		I^\circ_{g,\eta}\Big(z;\prod^{n}_{i=1}\tau_{k_i} \Big)&=z^{\sum_i k_i-\mathrm{vd}} \cdot \prod_j\frac{\eta_j^{\eta_j}}{\eta_j!} \cdot \int_{\overline{M}_{g,n+l(\eta)}} \frac{\Lambda^\vee(1)}{\prod_j (\frac{1}{\eta_j}-\psi_j)} \prod^n_{i=1} \psi_i^{k_i},
	\end{align*}
	where $\mathrm{vd} = 2g-2 + n + d + \ell(\eta)$. 
	Let
	\begin{align*}
		\mathrm{H}^\circ(w_1,\ldots,w_n;u) &:= \sum_{g\geq 0} u^{2g-2}\int_{\overline{M}_{g,n}}\frac{\Lambda^\vee(1)}{\prod_{i=1}^n (\frac{1}{w_i}-\psi_i)} \\
		\mathrm{H}(w_1,\ldots,w_n;u) &:= \sum_{\coprod_{i}S_i = \{1,\ldots,n\}} \prod_i \mathrm{H}^\circ(w_{S_i}; u). 
	\end{align*}
	Using the Hodge integrals above, we write numerical $I$-functions as follows,
	\begin{multline} \label{expressionI}
		I_{g,\eta}\Big(z; \prod^{n}_{i=1}\tau_{k_i} \Big) \\
		=z^{\sum_i k_i -\mathrm{vd} +1}  \cdot \prod_j\frac{\eta_j^{\eta_j}}{\eta_j!} [u^{2g-2}w_1^{k_1+1}\ldots w_n^{k_n+1}] \mathrm{H}(\eta_1,\ldots,\eta_l,w_1,\ldots,w_n;u),
	\end{multline}
	where the notation above indicates that we take the coefficient of a series at the term $u^{2g-2}w_1^{k_1+1}\ldots w_n^{k_n+1}$.  This will be useful for the operator expressions in Section \ref{operatorexpression}.

	\subsection{Numerical wall-crossing formula}
	
	A special role is played by numerical $I$-functions without markings. In fact, there exist only two  non-trivial numerical $I$-function without marked points.

	\begin{lemma} \label{empty} 
		Assume
		\begin{equation*}
			\deg_z I_{g,\eta}(z; \emptyset) \geq 0, 
		\end{equation*}
		then $g=2-d$ and $\eta=(2,1^{d-2})$, or $g=1-d$ and  $\eta=(1^d)$. In these cases, numerical $I$-functions are of the following form:
		\[
		I_{2-d,(2,1^{d-2})}(z; \emptyset)=1,\quad I_{1-d,(1^{d})}(z; \emptyset)=z. 
		\]
	\end{lemma}
	
	\begin{proof} Firstly, 
		\[
		\deg_z I_{g,\eta}(z; \emptyset) =  3-2g-d-l(\eta).
		\]
		Since the maps in $V_{g,0,\eta}$ consist of $\BP^1$ tubes possibly with contracted components, we must have 
		\[
		g\geq 1-\ell(\eta).
		\]
		Hence in order for $\deg_z (I_{g,\eta}(z;\emptyset))$ to be positive, we must have
		\[
		(g, \ell(\eta))=(2-d, d-1) \quad \text{or} \quad (1-d, d).
		\]
		In particular, $\eta = (2,1^{d-2}))$ or $\eta=(1^d)$. The last statement follows from a simple localisation calculation. 
	\end{proof}
	
	%\begin{definition} \label{stable}
	%	\textcolor{red}{Substitute everything with graphs} We define  a \textit{partition} of $(g,\{1,\dots, n\},d)$ of length $k$, 
	%	\[\vv{g}:=((g_1,N_1, \eta^{1}),\dots,(g_k,N_k,\eta^{k})),\] 
	%	such that $g_i \in \BZ$,  $\emptyset \neq N_i \subseteq \{1,\dots,  n\}$, and $\eta^{i}$ are  partitions of $d$, subject to the following conditions:
	%	\begin{itemize}
		%		\item[1)] $N_i \cap N_j=\emptyset, \quad \cup_i N_i =\{1,\dots,  n\},$
		%	\item[2)] $2g_i-2+d+\ell(\eta^{i})+n_i>0$.
		%	\end{itemize}
	%	\end{definition} 

\begin{corollary} \label{corollary1}
	We have
	\begin{multline*}
		\langle \tau_{k_1}(\gamma_1) \dots \tau_{k_n}(\gamma_n)  \rangle^{\mathsf{GW}}_{g} \\
		=\sum_{b, \Gamma} \Big\langle  I_{g_1,\eta^{1}}\Big( -\tilde{\psi}_1; \prod_{j\in N_1} \tau_{k_j}(\gamma_j) \Big) \dots I_{g_k,\eta^{k}}\Big( -\tilde{\psi}_k; \prod_{j\in N_k} \tau_{k_j}(\gamma_j) \Big), \mathbbm{1}^b \Big\rangle^{\mathsf{H}}_{\underline{\eta},(2)^b} \frac{1}{\Aut(\Gamma)\cdot b!},
	\end{multline*}
	such that the sum is taken over $b\in \BZ_{>0}$  and  graphs $\Gamma$ with $N_i\neq \emptyset$ for all $i\geq1$, subject to the following condition: 
	\begin{align*}
		g&=\sum^{k}_{i=1}(g_i+\ell(\eta^{i}))+ g_0-k, \quad &2g_0-2= d(2g(X)-2)+b+\sum^{k}_{i=1} (d-\ell(\eta^i)).
	\end{align*}

	%\begin{itemize}
	%\item $\eta^i$ are partitions of $d$ 
	%	\item $(N_1, \dots N_k)$ is an ordered partition of $\{1,\dots,n\}$, such that $N_i\neq \emptyset$, 
	%\item $g=\sum^{k}_{i=1}(g_i+\ell(\eta^{i}))+ g_0-k$, $  2g(X)-2=d(2g_0-2)+\sum^{k}_{i=1} (d-\ell(\eta^i))+b.$
	%	\end{itemize}	
\end{corollary}			

\begin{proof} This follows directly from Theorem \ref{maintheorem}, Lemma \ref{empty} and the projection formula. Insertion of $I_{2-d,(2,1^{d-2})}(z; \emptyset)$ corresponds to simple ramifications, hence the summation over integers $b$.    Since all source markings are on non-root vertices, the corresponding $\psi$-classes move to numerical $I$-functions.   Note that the numerical $I$-function  $I_{1-d,(1^{d})}(z; \emptyset)$ does not contribute, because it is unstable, i.e., it does not satisfy (\ref{stability}). 
\end{proof}

\begin{remark}Note that the condition $g=\sum^{k}_{i=1}(g_i+\ell(\eta^{i}))+ g_0-k$ is automatic, if all dimension constraints on the left and on the right of the equation above are satisfied. This follows from a simple dimension calculation. 
\end{remark}

\subsection{Wall-crossing for  point insertions}

Gromov--Witten invariants are the simplest when $\gamma_i=\omega$ for all $i$, where $\omega \in H^2(X)$ is the point class. This is because $\omega \cdot \omega=0$. 

\begin{corollary}\label{corollar2}
Let $\omega \in H^2(X)$ be the point class. We have 
\begin{equation*}
	\langle \tau_{k_1}(\omega)  \dots \tau_{k_n}(\omega) \rangle^{\mathsf{GW}}_{g}
	=\sum_{b,\Gamma} \Big\langle  I_{g_1,\eta^1}(-\tilde{\psi}_1; \tau_{1}(\omega)) \dots I_{g_n,\eta^n}(-\tilde{\psi}_n; \tau_{k}(\omega)), \mathbbm{1}^b \Big \rangle_{\underline{\eta},(2)^b}^{\mathsf{H}} \frac{1}{ b!},
\end{equation*}
such that the sum is taken over $b\in \BZ_{>0}$ and graphs $\Gamma$ with $N_i=\{i\}$ for all $i\geq 1$, 	subject to the same numerical condition as  in Corollary \ref{corollar2}. We order the non-root vertices of graphs according to the order of their leaves instead of genus labels. 
\end{corollary}
%\[ \sum_i (k_i-2g_i+2 -d - l(\eta^i) )=b. \]
\begin{proof}  Since $\omega \cdot \omega=0$ in $H^*(X)$, we have
\[ I_{g,\eta}\Big(z; \prod^{n}_{i=1}\tau_{k_i}(\omega)\Big)=0 \quad \text{ if } n>1.\] 
Therefore every $I$-function that occurs must have at most one marking. Since each vertex of  a graph has a distinct marked leaf, the automorphisms of such graphs are trivial. 
%The last statement follows from the dimension constraint of the moduli space of admissible covers; the dimension of $H_{\underline{\eta},(2)^b}(X)$ is $n+b$, while the degree of the insertion $I_{g_i,\eta^i}(-\tilde{\psi}_1, \tau_{k_1}(\omega))$ is $k_i-2g_i+2 -d - l(\eta^i)+1$. 
\end{proof}

\subsection{Further simplification of the wall-crossing formula for point insertions}
We simplify  Corollary \ref{corollar2} further by making the terms on the right-hand side more explicit. 
\begin{lemma} \label{raw} We have 
\[
\langle \tau_{k_1}(\omega)\ldots \tau_{k_n}(\omega) \rangle_{g}^\GW= \mathrm{Hur}^X(\tau_{k_{1}}^d,\dots, \tau_{k_n}^d ) 
\]
where 
\[
\tau_k^d=\sum_{\substack{b, \eta, \mu}} \Fz(\mu) \cdot \mathrm{Hur}^{\BP^1}\left(\mu,\eta,\frac{(-(2))^{b}}{b!}\right) \cdot I_{(k+2-b-d-\ell(\eta))/2,\eta}(\tau_{k}) \cdot (\mu)\in  \CZ(d).
\]
%Observe that the sum is finite because $I_{(k+2-b-d-\ell(\eta))/2,\eta}(\tau_{k})$ vanishes for big enough $b$. 
\end{lemma}
\begin{proof}
By Corollary \ref{corollar2}, we have 
\begin{equation} \label{theformula}
	\langle \tau_{k_1}(\omega)\ldots \tau_{k_n}(\omega) \rangle_{g,d}^\GW =\sum_{b,\Gamma} \Big\langle    \prod^{n}_{i=1} I_{g_i,\eta^{i}}\Big(-\tilde{\psi}_i; \tau_{k_i}(\omega) \Big), \mathbbm{1}^b \Big\rangle^{\mathsf{H}}_{\underline{\eta}, (2)^b}\frac{1}{b!}.
\end{equation}
By the string equation from Lemma \ref{string}, this is equal to
\[
\Big\langle   \prod_{i=1}^n (-1)^{b_i}\tau_{b_i}(\omega) \cdot  \mathbbm{1}^b \Big\rangle^{\mathsf{H}}_{\underline{\eta},  (2)^b} = (-1)^b\frac{b!}{b_1! \dots b_n!} \mathrm{Hur}^X(\underline{\eta}, (2)^b),
\]
where $b=b_1+\ldots+b_n$. As a result, we get 
\[
(\ref{theformula})=\sum_{b, \Gamma} \mathrm{Hur}^X\left(\underline{\eta}, (2)^b\right) \cdot \prod_{i=1}^n \frac{(-1)^{b_i}}{b_i!} I_{g_i,\eta^i}(\tau_{k_{i}}),
\]
where $b_i$ is the $z$-degree of $I_{g_i,\eta^i}(z; \tau_{k_i}(\omega))$, that is 
\[ 
b_i=k_{i}-2g_i+2 -d - l(\eta^i). 
\]
We then degenerate $X$ to $X$ with $n$ copies of $\p^1$ attached to it, which contain $b_i$ simple branching points and the branching point corresponding to $\eta^i$. Using the degeneration formula for Hurwitz numbers,
\[
(\ref{theformula}) =\sum_{\substack{ b, \Gamma, \underline{\mu}}} \mathrm{Hur}^X(\underline{\mu}) \cdot \prod_{i=1}^n \Fz(\mu^i) \mathrm{Hur}^{\BP^1}\left(\mu^i,\eta^i,\frac{(-(2))^{b_i}}{b_i!}\right)\cdot I_{g_i,\eta^i}(\tau_{k_{i}}).
\]
Expressing $g_i$ in terms of $b_i$, $\eta^i$ and $k_{i}$, we obtain the statement of the lemma. 
\end{proof}

\begin{remark}Observe that the expression of $\tau^d_k$ from Lemma \ref{raw} is similar to the one  obtained in \cite[Section 1.4]{OPcom} involving the Gromov--Witten invariants of the relative $\p^1$. In fact, our  $\tau^d_k$ can be obtained from the relative $\p^1$ by applying four operations simultaneously: 
	\begin{enumerate}
	 \item localisation on the relative $\p^1$, 
	 \item rigidification of the rubber $\p^1$, 
	 \item exchange of relative $\psi$-classes with absolute $\psi$-classes, 
	 \item  Gromov--Witten/Hurwitz correspondence for simple ramifications. 
	 \end{enumerate}
	 In particular, the statement presented in the lemma is equivalent to the one from \cite[Section 1.4]{OPcom}. However, we will use our $\tau^d_k$ to derive the Gromov--Witten/Hurwitz correspondence, as it admits a much nicer operator expression. 
\end{remark}

\section{Gromov--Witten/Hurwitz correspondence} \label{sectionGWH}

\subsection{Operator expressions} \label{operatorexpression}

In this section, we will evaluate $\tau_k^d$. We firstly express Hodge integrals and  therefore the numerical $I$-functions, and double Hurwitz numbers in terms of operators acting on the Infinite wedge $\wedge^{\frac{\infty}{2}} V$.  The necessary background for working with $\wedge^{\frac{\infty}{2}} V$ is recalled in Appendix \ref{Appendix}. 

Firstly,  the Burnside formula for Hurwitz numbers of a curve reads
\[
\mathrm{Hur}^{X}(\eta^1,\ldots,\eta^n) = \sum_{\lambda\vdash d} \left(\frac{\dim\lambda}{d!}\right)^{2-2g(X)}\prod_{i=1}^n f_{\eta^i}(\lambda),
\]
such that
\begin{equation} \label{flambda}
f_\eta(\lambda) = \frac{d!}{\Fz(\eta)}\frac{\chi_{\eta}^\lambda}{\dim\lambda}
\end{equation}
 for a character of the symmetric group $\chi^\lambda_\eta$ of an element in the conjugacy class $\eta$ acting on the representation associated to $\lambda$, and $\dim \lambda $ is the dimension of the representation.  For positive integers $\eta_i$, recall the operator expression for Hodge integrals, 
\[
\mathrm{H}(\eta_1,\ldots,\eta_n ;u) = u^{-l(\eta)-\lvert\eta\rvert}\prod_j \frac{\eta_j!}{\eta_j^{\eta_j}} \left\langle  e^{\alpha_1}e^{u\CF_2}\prod \alpha_{-\eta_j}\right\rangle,
\]which follows from expressing Hodge integrals in terms of Hurwitz numbers via the ELSV formula \cite{ELSV}, and then the Hurwitz numbers in terms of characters of symmetric groups via the Burnside formula; the operator expressions for the latter was  derived in \cite{Ok}. 
For arbitrary complex numbers $w_i$, Okounkov--Pandharipande \cite{OPeq} provide an extension of the formula above, 
\[ 
\mathrm{H}(w_1,\ldots,w_n;u) =u^{-n}\left\langle \prod \CA(w_i,uw_i)\right\rangle. 
\]
% Recall also that there is the following operator formula:
%	\begin{align*}
%	H(\eta_1,\ldots,\eta_l,w_1,\ldots,w_n;u) &= u^{-l-n}\left\langle \prod_{i=1}^n \CA(w_i,uw_i)\prod_{i=1}^l \CA(\eta_i,u\eta_i)\right\rangle\\
%	& = u^{-l-n-\lvert\eta\rvert}\prod_{i=1}^l \frac{\eta_i!}{\eta_i^{\eta_i}} \left\langle \prod_{i=1}^n \CA(w_i,uw_i)e^{\alpha_1}e^{u\CF_2}\prod_{i=1}^l \alpha_{-\eta_i}\right\rangle
%	\end{align*}
Using (\ref{expressionI}), we therefore obtain
\begin{multline}
I_{g,\eta}\Big(z; \prod^{n}_{i=1}\tau_{k_i} \Big) \\
= z^{\sum_i k_i -\mathrm{vd} +1} [u^{\mathrm{vd}}w_1^{k_1+1}\ldots w_n^{k_n+1}] \left\langle \prod \CA(w_i,uw_i)e^{\alpha_1}e^{u\CF_2}\prod \alpha_{-\eta_j}\right\rangle.
\end{multline}
%such that for $(g,n)=(-d+1,0)$ this expression is set to zero. 

Let us now list the formulas relevant for the proof of Theorem \ref{ProofGWH}. Firstly, the $I$-functions in the definition of $\tau_k^d$ involve only one $\psi$-class, which  by the preceding analysis admits the following expression:  
\begin{align*}
I_{g, \eta}(\tau_k) = \Fz(\eta) [u^{\mathrm{vd}}w^{k+1}] u^{-\ell(\eta)-d-1} \langle \eta| e^{u\CF_2}e^{\alpha_{-1}}\CA^*(w,uw) \rangle, 
\end{align*}
where we took the adjoint operators and 
\[
| \eta \rangle = \frac{1}{\Fz(\eta)}\prod_j \alpha_{-\eta_j} v_\emptyset.
\]
Furthermore, by the Burnside formula and (\ref{FB}),  we have
\[
\mathrm{Hur}^{\BP^1}(\mu,\eta, e^{-u(2)}) = \sum_{\lambda\vdash d}  \langle  \mu | e^{-\CF_2u}v_\lambda  \rangle\cdot \langle v_\lambda| \eta\rangle. 
\]

%\begin{multline*}
%	\langle \tau_{k_1}(\gamma_1) \dots \tau_{k_n}(\gamma_n)  \rangle^{\mathsf{GW}}_{g,d,\underline{\eta}}\\
%	=[u^{\sum_i (k_i+\deg_\BC(\gamma_i)-1)}]\sum_{\substack{\coprod_{i=1}^l S_i =\Set{1,\dots,n}\\S_i\neq\emptyset\\\eta^1,\dots,\eta^l\vdash d}} u^{\sum_i (d+l(\eta^i))} \Big\langle   \prod^{k}_{i=1} I_{\eta^{i}}\Big( \prod_{j\in N_i} \tau_{k_j}(\gamma_j) \Big) \cdot e^{u(2)I_{-d+2,(2)}}\Big\rangle^{+}_{d,\underline{\eta},(2)^b}
%	\end{multline*}
%	where the sum goes over unordered partitions of $\Set{1,\dots,n}$.

%\[
%\sum_{\eta,b}\frac{(-1)^b}{b!}\mathrm{Hur}^{\BP^1}(\mu,\eta,(2)^b) I_{(k+2-b-d-\ell(\eta))/2,\eta}\Big(\tau_k(\omega)\Big) = \sum_{\eta}[u^{k-d-\ell(\eta)}]\mathrm{Hur}^{\BP^1}(\mu,\eta,e^{-u(2)}) I_{\eta}\Big(\tau_k(\omega)\Big)
%\]

\subsection{Gromov--Witten/Hurwitz correspondence}
We write 
\[ \tilde{\mu}\leq \mu\] 
if $\mu$ can be obtained from $\tilde{\mu}$ by adding parts $\{1\}$, that is, if 
\[ \mu=(1^i, \tilde{\mu}), \quad \text{for some } i\geq0.\]
 We denote by $m_1(\mu)$ the multiplicity of $\{1\}$ in a partition $\mu$. The next result recovers the completed cycles of \cite{OPcom}. 
\begin{theorem}\label{ProofGWH} Consider $\tau^d_k$ from Lemma \ref{raw}. It admits the following expression, 
\[
\tau_k^d = \sum_{\substack{\mu\vdash d\\ \tilde{\mu} \leq \mu}} \binom{m_1(\mu)}{m_1(\tilde{\mu})}\cdot \rho_{k,\tilde{\mu}}\cdot(\mu) \in \CZ(d),
\]
where 
\[
\rho_{k,\mu} = \frac{\prod_j \mu_j}{\lvert\mu\rvert !}[x^{k+2-\lvert\mu\rvert-l(\mu)}] S(x)^{\lvert\mu\rvert-1}\prod_j S(\mu_j x).
\]

\end{theorem}

\begin{proof} Let 
\[	
I_{\eta}(u;\tau_k)=	\sum_{g} I_{g,\eta}(\tau_k) u^{2g-2}.
\]	
By operator expressions in Section \ref{operatorexpression} and the definition of $\tau_k^d$ from Lemma \ref{raw}, the coefficient of $\tau_k^d$ at $(\mu)\in \CZ(d)$ takes the following form:
\begin{align*}
[(\mu)]\tau^d_k&=\Fz(\mu)\sum_{\eta}[u^{k-d-\ell(\eta)}]\mathrm{Hur}^{\BP^1}(\mu,\eta,e^{-u(2)})\cdot I_{\eta}(u;\tau_k) \\
&= \Fz(\mu)[u^{k+1}w^{k+1}] \sum_{\eta} \sum_{\lambda\vdash d}   \langle \mu | e^{-\CF_2u}v_\lambda  \rangle \cdot \langle v_\lambda| \eta\rangle \Fz(\eta) \langle\eta\ | e^{u\CF_2}e^{\alpha_{-1}}\CA^*(w,uw) \rangle.
\end{align*}
This admits further simplification as follows,
\begin{align*}	
[(\mu)]\tau^d_k &= \Fz(\mu)[u^{k+1}w^{k+1}] \sum_{\lambda\vdash d} \langle \mu | e^{-\CF_2u}v_\lambda  \rangle  \cdot   \langle v_\lambda | e^{u\CF_2}e^{\alpha_{-1}}\CA^*(w,uw) \rangle\\
&=\Fz(\mu)[u^{k+1}w^{k+1}] \sum_{\lambda\vdash d}  \langle \mu  | v_\lambda\rangle \cdot \langle v_\lambda | e^{\alpha_{-1}}\CA^*(w,uw) \rangle\\
&=\Fz(\mu) [u^{k+1}w^{k+1}] \langle\mu|  e^{\alpha_{-1}}\CA^*(w,uw) \rangle\\
&= [u^{k+1}w^{k+1}] \left\langle\prod\alpha_{\mu_j} e^{\alpha_{-1}}\CA^*(w,uw) \right\rangle.
\end{align*}			
We slide $e^{\alpha_{-1}}$ to the left  by applying the bosonic commutation relation, and then apply  $
\left[\alpha_k,\CE_l(z)\right] = \varsigma(kz)\CE_{k+l}(z)$, this gives us that the expression above is equal to
\begin{align*}
& \hspace{0.4cm}\sum_{\tilde{\mu}\leq \mu} \binom{m_1(\mu)}{m_1(\tilde{\mu})}[u^{k+1}w^{k+1}] \left\langle\prod\alpha_{\tilde{\mu}_j}\CA^*(w,uw) \right\rangle\\
&= \sum_{\tilde{\mu}\leq \mu} \binom{m_1(\mu)}{m_1(\tilde{\mu})}[u^{k+1}w^{k+1}] S(uw)^w\frac{\varsigma(uw)^{\lvert\tilde{\mu}\rvert}}{(w+1)_{\lvert\tilde{\mu}\rvert}}\left\langle\prod\alpha_{\tilde{\mu}_j} \CE_{-\lvert\tilde{\mu}\rvert}(uw)\right\rangle\\
&= \sum_{\tilde{\mu}\leq \mu} \binom{m_1(\mu)}{m_1(\tilde{\mu})}[u^{k+1}w^{k+1}] S(uw)^w\frac{\varsigma(uw)^{\lvert\tilde{\mu}\rvert-1}}{(w+1)_{\lvert\tilde{\mu}\rvert}}\prod\varsigma(\tilde{\mu}_j uw).
\end{align*}
Now observe that the function involved in this expression  is of the form $f(uw,w)$ for some $f(x,y)$. Since we are interested in the coefficient of $u^iw^i$ in its Taylor expansion, we can take the specialisation $f(x,0)$.  This leaves us with
\begin{align*}
& \hspace{0.4cm}\sum_{\tilde{\mu}\leq\mu} \binom{m_1(\mu)}{m_1(\tilde{\mu})}[x^{k+1}] \frac{\varsigma(x)^{\lvert\tilde{\mu}\rvert-1}}{\lvert\tilde{\mu}\rvert!}\prod\varsigma(\tilde{\mu}_j x)\\
&= \sum_{\tilde{\mu}\leq \mu} \binom{m_1(\mu)}{m_1(\tilde{\mu})}\frac{\prod \tilde{\mu}_j}{\lvert\tilde{\mu}\rvert!}[x^{k+2-\lvert\tilde{\mu}\rvert-\ell(\tilde{\mu})}] S(x)^{\lvert\tilde{\mu}\rvert-1}\prod S(\tilde{\mu}_j x)
\end{align*}
as desired.
\end{proof}
\begin{remark}
Our expression of $\tau_k^d$ follows a slightly different convention.  While in \cite{OPcom}, the authors allow to insert partitions of $d'$ such that $d'\leq d$ by adding parts $\{1\}$ to a partition, we insert only partitions of $d$ with extra coefficients that depend on the number of parts $\{1\}$ in a partition. Also, our coefficients  $\rho_{k,\mu}$ are related to the coefficeints $\rho^{\mathrm{OP}}_{k, \mu}$ of \cite{OPcom}  by the formula $\rho_{k,\mu}=\rho^{\mathrm{OP}}_{k+1,\mu}/k!$
	\end{remark}
%We now try to get rid of the $e^{u\CF_2}$ factor after which we compute the vacuum expectation. For this we recall
%	\[
%\left[\CE_k(z),\CE_l(w)\right] = \varsigma\left(\det\left[\begin{smallmatrix}
%	k & z\\ l & w
%	\end{smallmatrix}\right]\right) \CE_{k+l}(z+w)
%	\]
%	and as a result we have 
%	\[
%	\left[\CF_2,\CE_k^{(n)}(w)\right] = [z^2]\left[\CE_0(z),\CE_k(w)\right] = -[z^2]\varsigma(kw)\CE_k(z+w) = -k\CE_k^{(n+1)}(w)
%	\]
%	and so
%	\[
%	e^{-u\CF_2}\CE_k(z)e^{u\CF_2} = \sum_{n\geq 0} (-k)^n\CE_k^{(n)}(z)\frac{u^n}{n!} = \CE_k(z+ku)
%	\]
%	hence also
%	\[
%	e^{-u\CF_2}e^{\alpha_1}e^{u\CF_2} = e^{\CE_1(u)}.
%	\]
%	This means that one can slide $e^{u\CF_2}$ all the way to the left and get
%	\[
%	\mu_{\eta}(\tau_k(\pt)) = \text{prefact}\cdot u^{-l(\eta)-d-1}\sum_{n=0}^d\frac{1}{n!}[z^k]S(uz)^z \frac{\varsigma(uz)^{d-n}}{(z+1)_{d-n}} \langle \CE_{d-n}((u+d-n)z) \CE_1(u)^n\prod_{i}\alpha_{-\eta_i}\rangle
%	\]
%	Using
%	\[
%	\left[\alpha_k,\CE_l(z)\right] = \varsigma(kz)\CE_{k+l}(z)
%	\]

\appendix

\section{Computational verification of Theorem \ \\ by Johannes Schmitt} \label{AppendixJohannes}

\subsection{Implementation and methodology}
This appendix presents computational evidence for Theorem \ref{maintheorem} in the case where the target is rubber $\mathbb{P}^1$ relative to two points, with no insertions $\gamma$ and trivial insertion $\alpha=1$. In this setting, the Gromov-Witten cycle equals the double-ramification cycle $\mathsf{DR}_g(A)$, where $A = (\mu_1, -\mu_2)$ with $\mu$ of length $m=2$ specifying the relative conditions over $0$ and $\infty$.

The computational verification was implemented using the \texttt{admcycles} package \cite{admcycles}, with code available at 
\begin{center} \href{https://gitlab.com/modulispaces/admcycles/-/tree/GW_admissible_correspondence}{gitlab.com/modulispaces/admcycles/-/tree/GW\_admissible\_correspondence}
\end{center}
The double-ramification cycles have been explicitly calculated following \cite{JPPZ}.
For the right-hand side of Theorem \ref{maintheorem}:
\begin{itemize}
	\item For maps of degree $d=1$, the  Hurwitz cycles $\mathsf{H}_{\underline{\mu},\underline{\eta}}$ are given by fundamental classes of the corresponding moduli spaces of stable curves.
	\item For maps of degree $d=2$, they are hyperelliptic cycles $\mathsf{Hyp}_{g,w+2, 2v }$ with  $w$ marked Weierstrass points and $v$ pairs of hyperelliptic conjugate points. These were calculated using the algorithm described in \cite{SvZ}.
\end{itemize}

\subsection{Results}

Table \ref{tab:verification} presents the computational verification results. For each genus $g$ and ramification profile $\underline{\mu}$, we computed both sides of the equation in Theorem \ref{maintheorem} and confirmed their equality. The table shows the number of non-zero summands in the graph sum on the right-hand side of the theorem, along with the rank of the tautological cohomology group $R^g(\Mbar^\circ_{g,n})$ where the equality holds.

\begin{table}[ht]
	\centering
	\begin{tabular}{cccc}
		\hline
		& & & \vspace{-0.4cm}  \\ 
			 $g$ & $\underline{\mu}$ & \# non-zero summands & rank $R^g(\Mbar^\circ_{g,n})$\vspace{0.1cm}  \\
		\hline 
		1 & $(1), (1)$ & 1 & 2 \\
		2 & $(1), (1)$ & 2 & 14 \\
		3 & $(1), (1)$ & 3 & 104 \\
		4 & $(1), (1)$ & 5 & 838 \\
		1 & $(2), (2)$ & 3 & 2 \\
		2 & $(2), (2)$ & 9 & 14 \\
		3 & $(2), (2)$ & 27 & 104 \\
		1 & $(2), (1,1)$ & 3 & 5 \\
		2 & $(2), (1,1)$ & 10 & 44 \\
		\hline
		\vspace{0.5cm}
	\end{tabular}
	\caption{Verification of Theorem \ref{maintheorem} for various genera and ramification profiles; for $(g,\underline{\mu}) = (3, (2), (2))$ some insertions at the root vertex of the were calculated manually in terms of $\tilde{\psi}_1 \cdot \pi_*( \mathsf{Hyp}_{g,w+2, 2v }) = \ \tilde{\psi}_1 \cdot \mathsf{Hyp}_{g,w'+2, 2v' }$ plus corrections using the projection formula, to avoid having to calculate the fully Weierstrass marked cycle, e.g.,  $\mathsf{Hyp}_{1,4, 4}$.}\label{tab:verification}\vspace{-0.2cm}
\end{table} 	

All cases verified the equality of the two sides of Theorem \ref{maintheorem}, providing strong computational evidence for the theorem. Many of these equalities occur in tautological cohomology groups of high rank, which makes their agreement particularly compelling.

\subsection{Computational challenges}

Several technical challenges were encountered during implementation:
\begin{itemize}
	\item Moduli spaces of potentially disconnected curves had to be translated into a combinatorial enumeration of all possible connected components.
	\item The insertions of $\psi$-classes at the root vertex must be performed before the (disconnected) gluing map, requiring the calculation of fully-marked hyperelliptic cycles $\mathsf{Hyp}_{g,w+2, 2v }$ whenever there is a $\psi$-insertion at root vertex. These cycles have high codimension and live on moduli spaces with many marked points, making calculations computationally intensive.
\end{itemize}

\subsection{Applications and future directions}

Since the double ramification cycle has an explicit graph-sum formula, Theorem \ref{maintheorem} can in principle be used to calculate admissible cover cycles with two fixed ramification profiles. In particular, this includes the hyperelliptic cycles $\mathsf{Hyp}_{g,2, 0 }$, which currently lack a graph sum formula valid for all genera $g$, though results exist for genus $g=2$ \cite{CT}.

To develop an effective recursion, one would need to express all hyperelliptic cycles $\mathsf{Hyp}_{g_0,w+2, 2v }$ for $g_0<g$ appearing in the root vertex of non-trivial graphs. For $w \leq 2$ and $v=0$, these would be known from previous cases of the recursion, but higher values of $w$ and $v$ would also be needed. These could be calculated iteratively using the formulas from \cite{EH}:
\begin{align*}
	\mathsf{Hyp}_{g_0,w+1, 2v } = D \cdot \pi^*(\mathsf{Hyp}_{g_0,w, 2v }) + (\textsf{boundary correction terms})
\end{align*}
where $D \in R^1(\Mbar^\circ_{g_0, w+1+2v})$ is an explicit divisor class and $\pi$ is the forgetful map of the last marked (Weierstrass) point. Similar formulas exist for increasing the number of hyperelliptic conjugate pairs.

This approach could be valuable both for concrete calculations of hyperelliptic cycles and for exploring graph summation formulas.
\section{Infinite wedge and operators} \label{Appendix}
\subsection{Infinite wedge} In this appendix we recall the necessary background for the operator formalism used in the proof of Theorem \ref{ProofGWH}.
The  source for the material presented here is \cite{OPeq, OPcom}, especially \cite[Section 2]{OPcom}.   

Let 
\[
V = \bigoplus_{k\in \BZ+\frac{1}{2}} \BC\cdot \underline{k}
\]
be a vector space with a basis indexed by half integers $\BZ +\frac{1}{2}$. 
We define
\[
\Lambda^{\frac{\infty}{2}} V = \bigoplus_{S} \BC\cdot v_S
\]
as the vector space with a basis consisting of infinite wedge products
\[
v_S = v_{s_1}\wedge v_{s_2} \wedge v_{s_3} \wedge \ldots 
\]
indexed by $S = \Set{s_1>s_2>\ldots }\subset \BZ+\frac{1}{2}$ satisfying:
\begin{itemize}
\item $S_{-} \coloneqq \{ s \notin S \mid  s \in \BZ_{\leq 0}-\frac{1}{2}\} $ is finite
\item $S_+ \coloneqq \{ s \in S \mid  s \in \BZ_{\geq0 }+\frac{1}{2}\}$ is finite.
\end{itemize}
This space carries an inner product by declaring
\[
(v_S,v_{S'}) = \delta_{S,S'}.
\]

\subsection{Zero charge} Let  $\Lambda_0^{\frac{\infty}{2}} V \subset \Lambda^{\frac{\infty}{2}} V$ be the subspace generated by $v_S$ of zero charge, that is, with $|S_-| = |S_+|$. Such $v_S$ are all of the shape
\begin{equation}
v_\lambda = \underline{\lambda_1 - \tfrac{1}{2}}\wedge \underline{\lambda_2-\tfrac{3}{2}}\wedge \underline{\lambda_3-\tfrac{5}{2}}\wedge \ldots
\end{equation}
for some partition $\lambda=(\lambda_1, \lambda_2, \hdots)$, $\lambda_1\geq \lambda_2 \geq \hdots \geq 0$. The vector associated to the empty partition is the vacuum vector, 
\[
v_\emptyset = \underline{-\tfrac{1}{2}}\wedge \underline{-\tfrac{3}{2}}\wedge\underline{-\tfrac{5}{2}}\wedge\ldots. \]
Let $A$ be an operator acting on $\Lambda_0^{\frac{\infty}{2}}V$, we define the expectation value of $A$ by 
\[ 
\langle A \rangle = (Av_\emptyset, v_\emptyset). 
\]

\subsection{Fermions and bosons} Let $\psi_k$ be the fermionic operator on $\Lambda^{\frac{\infty}{2}}V$ defined by
\[
\psi_k \left(v_S\right) = \underline{k}\wedge v_S, \quad k\in \BZ+\tfrac{1}{2},
\]
and let  $\psi^*_k$ be its adjoint with respect to the pairing $( \ , \ )$. We further define elementary matrix  operators $E_{i,j}$ on $\Lambda^{\frac{\infty}{2}}V$ by 
\[
E_{i,j} = \begin{cases}
\psi_i \psi_j^*, &j>0\\
-\psi_j^* \psi_i, &j<0.
\end{cases}
\]
Using $E_{i,j}$,  we define the  bosonic operators on $ \Lambda^{\frac{\infty}{2}}V$ by
\[\alpha_r= \sum_{k \in \BZ +\frac{1}{2}} E_{k-r, k}, \quad r \in \BZ.\]
Operators $E_{i,j}$ and $\alpha_r$ preserve the subspace $\Lambda_0^{\frac{\infty}{2}}V $, and $\alpha_r$ give rise to bosons associated a partition $\eta$, 
\[ |\eta \rangle= \frac{1}{\Fz(\eta)}\prod \alpha_{-\eta_j} v_\emptyset. \]
Bosons can expressed in terms of fermions using the  Murnaghan--Nakayama rule for characters of the symmetric group, 
\begin{equation} \label{FB}
|\eta \rangle= \sum_{\lambda\vdash |\eta|} \chi_\eta^\lambda v_\lambda. 
\end{equation}

\subsection{Operators} Let 
\begin{align*}
\varsigma(x) &= e^{x/2}-e^{-x/2}\\
S(x) &= \frac{\varsigma(x)}{x} = \frac{\sinh(x/2)}{x/2},
\end{align*}
and let 
\[
(a+1)_k = \frac{(a+k)!}{a!}=\begin{cases}
	(a+1)\dots(a+k), \text{ if } k\geq 0\\
	\left(a(a-1)\dots (a+k+1)\right)^{-1}, \text{ if } k< 0
	\end{cases}
\]
be the Pochhammer symbol. 
%\begin{cases}
%	(a+1)(a+2)\dots(a+k), \quad \hspace{0.08cm} \text{ if } k\geq 0\\
%	\left(a(a-1)\dots (a+k+1)\right)^{-1}, \text{ if } k\leq 0
%\end{cases}

We define the following operators on $\Lambda_0^{\frac{\infty}{2}}V$,
\begin{align*}
\CE_r(x) &= \sum_{k\in \BZ + \frac{1}{2}} e^{x(k-\frac{r}{2})} E_{k-r,k} + \frac{\delta_{k,0}}{\varsigma(x)} \\
\CA(a,b) &= S(b)^a \sum_{k\in \BZ}\frac{\varsigma(b)^k}{(a+1)_k} \CE_k(b)\\
\CF_2 &= \sum_{k\in \BZ+\frac{1}{2}} \frac{k^2}{2} E_{k,k}.
\end{align*}
Observe that the adjoint of $\CA(a,b)$ with respect to the pairing on $\Lambda_0^{\frac{\infty}{2}}V$ is of the following form:
\[
\CA^*(a,b) = S(b)^a\sum_{k\in\BZ} \frac{\varsigma(b)^k}{(a+1)_k}\CE_{-k}(b),
\]
and that $\CF_2 $ is diagonal with respect to the basis given by $v_S$, 

\[
\CF_2 v_S= f_2(S) v_S,
\]
where 
\[
f_2(S)=\sum _{ k\in S_+}\frac{k^2}{2} -   \sum _{ k\in S_-}\frac{k^2}{2},
\] 
which is equal to $f_{2}(\lambda)$ from (\ref{flambda}) for the partition $\lambda$ associated to $S$.

\bibliographystyle{amsalpha}
\bibliography{StableAdmissible}

%Furthermore, there is a formula for Hurwitz numbers, which reads:
%\[
%\mathrm{Hur}_g^{C}(\eta^1,\ldots,\eta^n) = \sum_{\lambda\vdash d} \left(\frac{\dim\lambda}{d!}\right)^{2-2g(X)}\prod_{i=1}^n f_{\eta^i}(\lambda)
%\]
%where 
%\[
%f_\eta(\lambda) = \frac{d!}{\Fz(\eta)}\frac{\chi_{\eta}^\lambda}{\dim\lambda}.
%\]

\end{document}